\documentclass{amsart}

\usepackage{amsthm,amsmath,amssymb,hyperref}
\newtheorem{theorem}{Theorem}[section]
\newtheorem{lemma}[theorem]{Lemma}
\newtheorem{proposition}[theorem]{Proposition}
\newtheorem{corollary}[theorem]{Corollary}
\newtheorem{remark}[theorem]{Remark}

\theoremstyle{definition}
\newtheorem{definition}[theorem]{Definition}

\numberwithin{equation}{section}

\newcommand{\Dast}[1]{\Delta^{\ast}(#1)}
\newcommand{\mD}[1]{\min\Delta(#1)}
\newcommand{\In}[1]{I(#1)}

\newcommand{\Hud}[1]{J(#1)}

\newcommand{\ko}{\mathsf{k}}
\newcommand{\Ko}{\mathsf{K}}
\newcommand{\vo}{\mathsf{v}}
\newcommand{\ro}{\mathsf{r}}
\newcommand{\Zo}{\mathsf{Z}}
\newcommand{\Lo}{\mathsf{L}}
\newcommand{\qo}{\mathsf{q}}
\newcommand{\ao}{\mathsf{a}}
\newcommand{\bo}{\mathsf{b}}

\newcommand{\ac}{\mathcal{A}}
\newcommand{\bc}{\mathcal{B}}
\newcommand{\fc}{\mathcal{F}}
\newcommand{\lc}{\mathcal{L}}
\newcommand{\ic}{\mathcal{I}}
\newcommand{\cc}{\mathcal{C}}
\newcommand{\dc}{\mathcal{D}}
\newcommand{\pc}{\mathcal{P}}
\newcommand{\oc}{\mathcal{O}}

\newcommand{\s}{\sigma}

\newcommand{\N}{\mathbb{N}}
\newcommand{\Q}{\mathbb{Q}}
\newcommand{\Z}{\mathbb{Z}}

\newcommand{\red}{\textrm{red}}

\DeclareMathOperator{\ord}{ord}
\DeclareMathOperator{\lcm}{lcm}
\DeclareMathOperator{\supp}{supp}

\begin{document}
\title[Congruence half-factorial Krull monoids]{On congruence
half-factorial Krull monoids with cyclic class group}
\thanks{Both authors were supported by the ANR Caesar grant number ANR-12-BS01-0011. Parts of the work were carried out while W.S. was supported by the Austrian Science Fund (FWF): J~2907-N18}

\author{A. Plagne \and W.~A. Schmid}



\email{alain.plagne@polytechnique.edu}
\email{schmid@math.univ-paris13.fr}
\address{(A.P.) Centre de Math\'ematiques Laurent Schwartz,
\'Ecole polytechnique,
91128 Pa\-laiseau Cedex,
France}
\address{(W.A.S.) Universit{\'e} Paris 13, Sorbonne Paris Cit{\'e}, LAGA, CNRS, UMR 7539,  Universit{\'e} Paris 8, F-93430, Villetaneuse, France, and  Laboratoire Analyse, G{\'e}om{\'e}trie et Applications (LAGA, UMR 7539), COMUE  Universit{\'e} Paris Lumi{\`e}res,  Universit{\'e} Paris 8, CNRS, 93526 Saint-Denis cedex, France}

\begin{abstract}
We carry out a detailed investigation of congruence half-factorial Krull monoids with finite cyclic class group and related problems. Specifically, we determine precisely all relatively large values that can occur as a minimal distance of a Krull monoid with finite cyclic class group, as well as the exact distribution of prime divisors over the ideal classes in these cases. Our results apply to various classical objects, including  maximal orders and certain semi-groups of modules. In addition, we present applications to quantitative problems in factorization theory. More specifically, we determine exponents in the asymptotic formulas for the number of algebraic integers whose sets of lengths have a large difference.  
\end{abstract}

\maketitle

\section{Introduction}
\label{sec_int}

This paper is concerned with two very closely linked questions. On the one hand, we study properties of Krull monoids with finite cyclic class group that have a specific arithmetic property. On the other hand, we apply these results to obtain a refined understanding of the arithmetic of Krull monoids with finite cyclic class group where each class contains a prime divisor. There are a variety of classical structures to which our results apply, including  maximal orders of number fields (more generally holomorphy rings of global fields) with finite cyclic class group, certain semigroups of isomorphy classes of modules, certain Diophantine monoids and several others (for details we refer to Section \ref{ssec_km}).  
The arithmetic property in question is the notion of congruence
half-factoriality. We recall that an atomic monoid -- in
this paper, the term monoid always means a commutative cancellative
semigroup with identity, a classical example is the multiplicative monoid of a domain -- is a monoid such that each
non-zero and non-invertible element is the product of irreducible
elements. The monoid is called factorial if each element has an essentially
unique factorization into irreducibles, where by essentially unique we mean unique up to ordering and associates.
If only the number of factors in the factorizations of an element is uniquely determined by the element,
the monoid is called half-factorial.
This property was first investigated by Carlitz \cite{carlitz60};
a first more systematic investigation of this property was carried out by Skula \cite{skula76}, {\'S}liwa \cite{sliwa76}, and Zaks \cite{zaks76}
motivated by some number-theoretic questions of Narkiewicz (see \cite[Chapter 9]{narkiewicz04}). 

Roughly two decades ago, on the one hand Chapman and Smith \cite{chapmansmith90a,chapmansmith90} 
considered another (weakened) condition along these lines. Namely, they considered the property that the number of factors in a factorization  (also called the length of the factorization) of an element is
only unique modulo some fixed non-negative integer $d$. They called this property $d$-congruence half-factoriality. Of course,
$0$-congruence half-factorial means merely half-factorial
whereas the condition $1$-congruence half-factorial is void. However,
for any other choice of $d$, this is a new and non-trivial condition.
On the other hand, Geroldinger \cite{geroldinger88} started a
systematic investigation of the structure of set of lengths of factorizations.
We informally recall the definition of a set of lengths.
For a non-zero and non-invertible element $a$ of an atomic structure one defines the set of lengths of $a$, denoted by $\Lo(a)$ as the
set of all integers $\ell$ such that there exist irreducible elements
$u_1, \dots, u_{\ell}$ such that $a=u_1 \dots u_{\ell}$; for an
invertible element, its set of lengths is $\{0\}$ and for the zero-element it
is empty. Geroldinger \cite{geroldinger88,geroldinger90b} proved that for a fixed structure
(of the above form) there is a
finite set of positive integers such that all the sets of lengths are almost arithmetical multiprogressions
(i.e., a certain union of arithmetic progressions
with the same difference from which some elements at the `beginning'
and the `end' might be removed) with a difference that is an element
of this set and there is a global bound on the number
(and location) of the `removed' elements (see Section \ref{ssec_kn} for a precise definition).
To get an understanding of the differences appearing in this
description it is necessary to understand the minimal distance between
the elements of the sets of lengths of certain
submonoids (even for domains one has to consider submonoids).

These two very closely related themes suggest two pairs of questions.

\begin{itemize}
\item For a given type of class group what are the values of $d$
such that $d$-congruence half-factorial Krull monoids with this type of class group exist?  
Essentially equivalently, what are the minimal distances in sets of lengths appearing for Krull monoid with this type of class group?
\item For such a $d$, what are conditions on the monoid that
characterize that it is $d$-congruence half-factorial? Essentially
equivalently, which monoids do yield a specific minimal distance? 
\end{itemize}

Since the introduction of these ideas, both these questions were
investigated by various researchers (for recent contributions see, e.g., \cite{changetal07,gaogeroldinger00,geroldingerhamidoune02,geroldingerzhongDast,geroldingerzhongCHAR,WAS4,zhongRJ}). The investigations so far mainly focused on the first type of questions,
due to the fact that a (partial) solution to it is a precondition for
even beginning to consider the second one; see \cite{geroldingerzhongDast,WAS16} for some initial results. 
Moreover, we point out that it was typical -- we do so as well  -- to focus on (relatively) large $d$; mainly, since they are the more interesting ones in understanding the arithmetic and since they are more relevant in applications, for example, to the problem of giving arithmetic characterizations of the class group (see, e.g., \cite{geroldingerzhongCHAR,WASchar,WAScharR2,zhongRJ}).

By transfer results -- a first version in the number
theoretic context is due to Narkiewicz \cite{narkiewicz79}, later
developments are mainly due to Geroldinger and Halter-Koch, we refer
to their monograph \cite[Section 3.2]{geroldingerhalterkochBOOK} for an
overview -- it is known that all these questions depend only on the
distribution of prime divisors or prime ideals (in a suitable sense) over the classes, and that they can be studied in the associated block monoid, that is, a monoid of zero-sum sequences over the class group.

The investigations of this paper are restricted to the case where the
class group of the underlying monoid is finite cyclic, which arguably is a central case 
that already received considerable attention.
While our results do not provide a complete answer to the two
types of questions at hand -- we present some arguments why we consider
it as highly unlikely that such a complete answer will be found in the
foreseeable future -- we present results that go significantly beyond
what was previously known.
In particular, this allows us to explore certain phenomena that
essentially were `invisible' in all situations considered so far.
In the following section, we informally discuss part of our results and the context.

\section{Overview of results and methods}

As mentioned in the Introduction, we study several closely linked questions that are not exactly identical, 
yet to a considerable extent reduce to the same core problem. We start by discussing our contribution to the core problem and then some applications of these results.

We only recall  precise definitions in the next section (see in particular Section \ref{ssec_kn}).
This core problem is to understand for a finite abelian group $G$ with $|G| \ge 3$ -- in this paper we focus on cyclic groups  -- the set of minimal distances  $\Dast{G}$ and associated inverse problems. The restriction $|G| \ge 3$ is due to the fact that by a well-known result, an early version is due to Carlitz \cite{carlitz60}, for $|G|\le 2$ one only gets half-factorial structures where the problems that concern us here do not arise.

This set $\Dast{G}$ is a finite set of positive integers, namely the integers $\min \Delta (G_0)$ where $G_0 \subset G$ is a subset, fulfilling a certain non-degeneracy condition. And, $\Delta(G_0)$ is the set of successive distances of the monoid of zero-sum sequences over $G_0$, i.e., distances between successive elements in the sets of lengths of these sequences.
By the associated \emph{inverse problem} we mean the problem of determining, for some element $d \in \Dast{G}$, the structure of all sets $G_0\subset G$ such that $\min \Delta(G_0)=d$.

As is common we focus on the problem for large elements of $\Dast{G}$; where large is to be understood as still relatively large in comparison to $\max \Dast{G}$.

First, we briefly recall what was known on this problem;
needless to say, this is not intended as a historical survey, we
merely wish to provide some context for our results. We focus on the case that $G$ is finite cyclic.

By a very recent result of Geroldinger and Zhong \cite{geroldingerzhongDast} it is known that for $G$ a finite abelian group, with $|G| \ge 3$,  
\[\max \Dast{G} =  \max \{\exp(G)-2, \ro(G) - 1\}\]
where $\exp(G)$ and $\ro(G)$ denote the exponent and the rank of $G$, respectively (for definitions see 
the subsequent section). 
In the other direction, it is known that $\min \Dast{G}=\mD{G}=1$ for $|G|\ge 3$ (see \cite[Theorem 6.7.1]{geroldingerhalterkochBOOK}).

If $G$ is a finite cyclic group, by a result of Geroldinger and Hamidoune
\cite{geroldingerhamidoune02}
(for earlier results see \cite{gaogeroldinger00} and \cite{geroldinger90}),
we have
\begin{equation}
\label{GeroHami}
\max \Dast{G} = |G|-2 
\quad \text{ and } \quad 
\max (\Dast{G}\setminus \{|G|-2\}) = \left\lfloor \frac{|G|-2}{2} \right\rfloor.
\end{equation}
That is, also the second largest element is known; for results along these lines for 
more general groups we refer to \cite{geroldingerzhongDast,geroldingerzhongCHAR,WASchar}.
Moreover, the structure of all sets corresponding to the maximal
element $|G|-2$ is known; this was proved in \cite{WAS16} for a larger
class of groups, not only  cyclic ones.

There are also results showing that certain numbers are elements of
$\Dast{G}$ that are all based on the fact that the situation where only
two non-zero classes contain prime ideals is well understood due to
the works \cite{geroldinger90b} and \cite{changetal07} (see
Theorem \ref{aux_thm_2elements} and the discussion there).

Conversely, for cyclic groups of prime power order, yet not in the
general case, these results were applied (see \cite{geroldinger90b}) to obtain `upper bounds' for
$\Dast{G}$ via reducing to the two-elements case and using a certain
divisibility property; see equality  \eqref{eq_min=gcd} for this property 
and the discussion below for further details.

Here, for general finite cyclic groups, of sufficiently large order,  we determine all the elements of $\Dast{G}$ of size at least $|G|/10$. 
Thus, we considerably improve the above-mentioned results \eqref{GeroHami}.

For illustration we state a weakened version of our main result
going only down to $|G|/5$; for stating our actual result (Theorem \ref{main_thm_direct}),
we use some specialized notation, which we do not want to state
right away.

\begin{theorem}
\label{thm_directweak}
Let $G$ be a finite cyclic group of order at least $250$. 
We have
\begin{equation}
\begin{split}
\Dast{G} \cap \N_{\ge |G|/5}  =
\N \cap \biggl \{ |G|- & 2, \frac{|G|-2}{2}, \frac{|G|-3}{2}, \frac{|G|-4}{2},
\frac{|G|-4}{3},   \\
 &  \frac{|G|-6}{3},\frac{|G|-4}{4},\frac{|G|-5}{4},\frac{|G|-6}{4},
\frac{|G|-8}{4} \biggr \}.
\end{split}
\end{equation}
\end{theorem}

While in retrospect there is a clear explanation why exactly these
values and no others appear, the precise structure is somewhat subtle.

The condition that $|G| \ge 250$ should be essentially purely technical
and the value $250$ is chosen generously -- if one wished to be
precise regarding this bound one would have to compute a bound for each
value (see Remark \ref{rem_sizen0}). 
It is  clear that one needs some bound to avoid degenerate cases,
yet our bound is admittedly larger than the one necessitated by this
effect. It stems from the fact that we have a `rough' description for all
values of $\Dast{G}$ as small as $(2 |G|^2)^{1/3}$ (cf. Theorems \ref{2maximal}), and we need that the range in which we want to make
precise statements is `above' this value.
While we do not know how to significantly improve the $(2 |G|^2)^{1/3}$
in general, we point out that for a specific value of $|G|$ this bound
is improvable by a direct calculation and, if needed, the methods
presented in this paper are strong enough to allow to reduce the bound
to the threshold where the result itself becomes unfeasible -- yet
doing so is likely very
tedious.

Moreover, we show that in case the order of $G$ is a prime power, this `rough' description can be turned 
into a precise one in the full range, i.e., we obtain an exact description of all elements of $\Dast{G}$ of 
size at least $(2|G|^2)^{1/3}$ (see Theorem \ref{thm_pgroups}), improving on the above-mentioned `bounds'; 
the result is particularly interesting due to the fact that as in the general case, but even more visibly, one gets 
that certain yet not all divisors of elements whose presence in $\Dast{G}$ was already known are contained in $\Dast{G}$.

The above-mentioned results give a (partial) answer to the first type
of questions, i.e., what are the large orders for which congruence
half-factorial Krull monoids exist and what are
the (large) differences in their sets of lengths.

Regarding the second type of questions, we provide complete answers for
all the values listed in Theorem \ref{thm_directweak}
and some selected additional values for general finite cyclic $G$ and we provide answers for
considerably more values in the case where the cardinality of $G$ is a prime number. 
Here is the result we obtain in this special case and  below we give a partial version of the result in the general case.

\begin{theorem}
\label{thm_caseprime}
Let $H$ be a non-half-factorial Krull monoid with finite cyclic class group $G$ and suppose the order of $G$ is prime. 
Let $d \ge (2|G|^2)^{1/3}$. The monoid $H$ is $d$-congruence half-factorial if and only if at most two non-zero classes, 
say $g_1,g_2$, contain prime divisors and these two classes satisfy the property that there are two integers $c_1,c_2 \in \N$
such that $c_2g_1 + c_1 g_2 = 0$ and  $dc_1c_2$ divides $|G|-c_1- c_2$.
\end{theorem}
Note that the last condition implies  that $(|G|-c_1- c_2)/(c_1c_2)$ is integral.  

The fact that structures fulfilling the respective conditions on the classes are $d$-congruence half-factorial was already known by the work of Chang, Chapman and Smith \cite{changetal07}, and moreover they showed that under the assumption that only two non-zero classes contain prime divisors these are all (indeed they proved this with a weaker condition on $d$, cf.~Theorem \ref{2maximal} for details).

We point out that for the case of general $|G|$ to obtain such results,
even only a complete characterization for all values $d$
that are at least $f(|G|)$  for any $f$ that is $o(|G|)$, presently seems
completely unfeasible.
Doing so would require -- yet most likely would not suffice -- to first
obtain a complete understanding of half-factorial sets of all finite cyclic
groups. To see that this is the case it suffices to recall that 
for $G_0'$ a half-factorial subset of a finite abelian group  $G'$ and $G_0''$ a subset of a finite abelian group group $G''$ 
the set of distances of $(G_0' \cup G_0'') \subset G' \oplus G''$ is equal to the set of distances of $G_0''$.  

Now, if we have $|G|= mn$ with co-prime $m$ and $n$ the cyclic group $G$ is  the direct 
sum of a cyclic group of order $m$ and a cyclic group of order $n$.
For $m$ arbitrary and  $n$ large relative to $m$, the union 
of  a set with minimal distance $n-2$ from the cyclic group of order $n$ with a half-factorial 
set of the group of order $m$ would still have minimal distance $n-2$, which would 
exceed $f(mn)$ for sufficiently large $n$. Thus, in order to determine all such sets we would need to determine all half-factorial subsets of a cyclic group of order $m$ for arbitrary $m$.   

The problem of determining the half-factorial sets of finite cyclic
groups is studied since the mid-Seventies, yet the
answers obtained so far are not at all complete (see \cite{WAS12} for
partial results).
Thus, one is limited to a result addressing the problem for $d\ge c|G|$
for some $c<1$. The precise choice regarding $c$ we made, that is
 $1/5$, is somewhat arbitrary. 
 On the one hand, we wanted to choose it small enough to highlight certain interesting phenomena, and on the other hand, we did not want to choose
it too small as the technical difficulty increases rapidly. Our choice
is not the limit of our method.

Now, we only state the result for $d$ equal to $|G|-2$, $(|G|-2)/2$, $(|G|-3)/2$ and $(|G|-4)/2$.

\begin{theorem}
\label{thm_genprel}
Let $H$ be a non-half-factorial Krull monoid with finite cyclic class group of order $n \ge 250$. Then,
\begin{enumerate}
\item $H$ is $(n-2)$-congruence half-factorial if and only if exactly two non-zero classes, say $g_1$ and $g_2$, contain prime divisors, they are of order $n$, and they satisfy $g_1 = -g_2$,
\item $H$ is $(n-2)/2$-congruence half-factorial if and only  if exactly two non-zero classes of order $n$, say $g_1$ and $g_2$, contain prime divisors and they satisfy $g_1 = -g_2$ and the only other non-zero class that might contain prime divisors is the class of order $2$,
\item $H$ is $(n-3)/2$-congruence half-factorial if and only
if exactly two non-zero elements of the class group, say $g_1$ and $g_2$, contain prime divisors, they are of order $n$, and they satisfy $g_1=-2g_2$ (or $g_2 = -2g_1$),
\item $H$ is $(n-4)/2$-congruence half-factorial  if and only
if 
\begin{itemize}
\item a class $g_1$ of order $n$ contains prime divisors as does the class $-2g_1$, and the only other non-zero class that might contain prime divisors is $2g_1$, or
\item exactly three non-zero classes contain prime divisors, two are of order $n/2$, say $g_1$ and $g_2$, and satisfy $g_1 = -g_2$, and the third one is of order $2$, and $4 \nmid n$. 
\end{itemize}
\end{enumerate}
\end{theorem}

For the full result see Theorem \ref{char_thm}; yet, this is phrased in an abstract way, and for a more direct but slightly imprecise impression see Theorem \ref{main_thm_inverse} -- note that the sets do not match exactly in one case of point (iv) for reasons explained in Section \ref{sec_appCHF}.

In addition to groups of prime order, for other special types of cyclic groups -- such as those whose order is a prime-power, a product of two prime powers, or that has only `large' prime divisors -- various variants of our results could be
established more or less directly with the methods at hand. We do not
state too many of them explicitly; rather we try to present the approach
developed in this paper in as modular and general a form as seemed
feasible, in order to make future applications and modifications, if
needed, as simple as possible; we expect that certain variants will be
relevant to future investigations, it is however impossible to foresee,
which exact variant will be relevant.

We end by outlining the key points of our approach to this problem, and the general structure of the paper.

First, we recall some standard notions (see Section \ref{sec_prel}), and collect together and establish various, in part technical, results that are used throughout the paper (see Section \ref{sec_aux}). This includes some of the already mentioned results by Geroldinger \cite{geroldinger90b}, and by Chang, Chapman, and Smith \cite{changetal07, changetalother} as well as some results that can be seen as part of the recently introduced auxiliary framework of higher-order block monoids (see \cite{WAShigherorder}), generalizing the classical block monoid context. 
These last results are relevant for composite $n$, since in this case it is a significant problem that elements of various different orders exist, and to a certain extent they can be resolved using these tools.

Second, we establish the following key technical result which gives an at first seemingly weak conclusion, which however is
crucial in order to have the tools mentioned above at hand (for details on undefined notation see Section \ref{sec_prel}).

\begin{theorem}
\label{prop_2subset}
Let $G$ be a finite cyclic group, with $|G| \ge  3$, and let $G_0\subset G$
such that $\min \Delta(G_0)\ge \log |G|$. Then there exists a subset
$G_2\subset G_0$ of cardinality $2$ that is not half-factorial.
\end{theorem}

It is our understanding, that the lack of such a result impeded earlier progress on the problems considered in this paper. 
Its proof is embedded into a more general analysis of the relation between weakly half-factorial sets, large-cross-number sets and half-factorial sets, for finite cyclic groups and more generally torsion groups. This is carried out in Section \ref{non-LCN}.

Then, after having applied all this machinery, a variety of results
are already established (in a conceptual way); however, to obtain
results of the announced strength we need to carry out various quite
particular investigations, which as said we did not carry out to the
absolute limit of what seemed manageable to us, but we rather tried to
get a balance between added insight and technical complexity.
Section \ref{sec_spec} contains preparatory results, in Lemma \ref{main_lem_ul} we collect the facts obtainable without these particular investigations, and the remainder of Section \ref{sec_proofs} is dedicated to obtain our results on $\Dast{G}$ and the associated inverse problem.

Finally, in Section \ref{sec_appCHF} we apply these results to obtain results on congruence half-factoriality. In part this application is, by known results, very direct; however, for other parts some work is required. And, in Section \ref{sec_nt} we give applications to quantitative questions of factorization theory.
More specifically, it is a classical problem in factorization theory, indeed one of the motivating problems (see \cite[Chapter 9]{narkiewicz04}) to obtain asymptotics on the number of (non-associated) elements of a ring of algebraic integers having some specific factorization property; an utmost classical example being that the element is prime, however in particular in view of the fact that factorizations are non-unique there are numerous variants.
Now, as mentioned in the Introduction, the sets of lengths of the elements of a  ring of algebraic integers are suitably generalized arithmetic progressions, in particular there is a finite number of `patterns' (the technical term is period) that describe up to small deviations all sets of lengths. An important goal in the present context is to count the number of elements corresponding to one fixed such `pattern'.
It is known that the order of the respective counting function is
\[
\frac{x}{(\log x)^{\alpha}}(\log \log x)^{\beta}
\]
where the real numbers $\alpha$ and $\beta$ depend only on the `pattern' and the class group.
While there is an abstract combinatorial description of these constants known (see Section \ref{sec_nt} for details), 
there are only very few and special cases in which that combinatorial  problem was solved and the numerical values of these constants are known (see \cite[Theorem 9.4.10]{geroldingerhalterkochBOOK} 
and \cite{WAS16}). Using our result, we can determine the value of $\alpha$ in several new cases.

\section{Preliminaries}
\label{sec_prel}

To fix notations and to provide some background, we briefly recall
some key notions of this paper. Our notation follows closely the one
frequently used in factorization theory; for an expansive
account see \cite{geroldingerhalterkochBOOK}, and for an introduction
to the aspects most relevant to the current paper see \cite{geroldinger_lecturenotes}.

\subsection{Generalities}

We denote by $\N$ and $\N_0$ the positive and non-negative integers, respectively. Moreover, we use for $a$ a real number notation like $\N_{>a}$ and $\N_{\ge a}$ with the obvious meaning. Occasionally, it is convenient to have the convention that  $\gcd \emptyset = \min \emptyset = 0$.
All intervals in this paper are intervals of integers, even if the end-points are non-integral.

\subsection{Monoids and factorizations}

In this paper, we say that $H$ is a \emph{monoid} if it is a commutative, cancellative semigroup with identity; we use
multiplicative notation for monoids and denote the identity element
simply by $1_H$ or just $1$ (if there is no ambiguity).
The classical example to have in mind is the multiplicative monoid of non-zero elements of a domain.
The monoid $H$ is called \emph{atomic} if each non-invertible element is the product of
(finitely) many irreducible elements (also called, atoms).
We denote the set of irreducibles of $H$ by $\ac(H)$, and the set of
invertibles by $H^{\times}$. Moreover, the monoid is called \emph{reduced} if
$H^{\times}=\{1\}$; we denote by $H_{\red}= H/H^{\times}$ the reduced
monoid associated to $H$. Recall that $H$ is atomic if and only if $H_{\red}$ is atomic.

A submonoid $H'\subset H$ is called \emph{divisor closed} if for each $a\in H'$
every $b \in H$ with $b \mid_H a$ is already
an element of $H'$; and it is called \emph{saturated} if for $a,b \in H'$
one has $a \mid_{H'} b$ if (and only if) $a\mid_H b$, in other words if $b=ac$ for some $c\in H$ then in fact $c\in H'$. If the monoid in question is clear we omit the subscript specifying in which monoid the divisibility relation holds.

For a set $P$, let $\fc(P)$ denote the free (abelian) monoid with basis $P$.
And, for $p\in P$ and $f\in \fc(P)$, let $\vo_p(f)\in \N_0$ denote the
$p$-adic valuation of $f$ (or the multiplicity of $p$ in $f$), so that
\[f = \prod_{p \in P} p^{\vo_p(f)};\]
of course, all but finitely many of the $\vo_p(f)$ are $0$.
To preserve certain helpful connotations, we frequently refer to
elements of $\fc(P)$ as \emph{sequences} over $P$; in particular, we do so if
$P$ is (a subset of) an abelian group (cf.~below). And, in the same
vein, we refer to $\sum_{p\in P}\vo_p(f)$ as the \emph{length} of $f$, which
we denote by $|f|$, and to the identity element as the empty
sequence.

Since we occasionally make use of it, we recall a formal framework for
studying factorizations. For a monoid $H$, let
\[\Zo(H) = \fc(\ac(H_{\red}))\]
denote the monoid of factorizations of $H$,
and let
\[\pi_H : \Zo(H) \to H_{\red}\] denote the factorization homomorphism,
i.e., the homomorphism defined via $\pi(a)=a$ for $a\in \ac(H_{\red})$.

For $a\in H$, let
\[\Zo_H(a)= \pi_H^{-1}(aH^{\times})\]
denote the set of factorizations of $a$ (in $H$). Again, we typically
omit explicit references to $H$. A monoid is \emph{factorial} if and only if
the set of factorizations of each element is a singleton; and it is
\emph{atomic} if and only if the set of factorizations of each element is
nonempty.
Moreover, let
\[\Lo_H(a) = \{|z|\colon z \in \Zo_H(a)\}\]
denote the \emph{set of lengths} of $a$.
Equivalently, for $a\in H\setminus H^{\times}$, the set $\Lo(a)$ is
the set of all $\ell \in \N$ such that there exist $u_1, \dots ,
u_{\ell} \in \ac(H)$ with $a = u_1\dots u_{\ell}$, i.e., $a$ has a
factorization into irreducibles of length $\ell$. For $a\in
H^{\times}$ we have $\Lo(a)=\{0\}$.

An important notion in factorization theory is that of a
\emph{transfer homomorphism}. A homomorphism $\theta: H \to B$ between
monoids is called a transfer homomorphism if
\begin{itemize}
\item the product set
$\theta(H)B^{\times}$ is equal to $B$ (in other words, $\theta$ is surjective up to invertible elements),
\item $\theta^{-1}(B^{\times}) = H^{\times}$, and 
\item for each $u\in H$ and $b_1,b_2\in B$ with $\theta(u) = b_1 b_2$
there exist $u_1,u_2\in H$ such that $u=u_1u_2$ and $\theta(u_i)$ and
$b_i$ are associated (i.e., are equal up to multiplication with an
invertible element) for $i=1$ and $2$.
\end{itemize}

It is easy to see that a transfer homomorphism $\theta: H \to B$ induces a transfer homomorphism of the respective associated reduced monoids, which we also denote by $\theta$.
A key property of a transfer homomorphism is that it also induces a
homomorphism of the factorization monoid $\Zo(H)$ and $\Zo(B)$, given by $u_1 \dots u_{\ell} \mapsto \theta(u_1) \dots \theta(u_{\ell})$; note that atoms are mapped to atoms. So, in particular $\Lo_H(a)=\Lo_B(\theta(a))$ for each $a\in H$.

Finally, we recall the definition of the \emph{set of distances} of a monoid.
First, for a set $L= \{\ell_1, \ell_2, \dots \}$ of integers with
$\ell_i < \ell_{i+1}$ for each index $i$, we let
\[
\Delta(L)= \{\ell_2 - \ell_1, \ell_3 - \ell_2, \dots \}
\]
denote the set of successive distances of $L$.
Now, for $a\in H$, we define $\Delta(a)= \Delta(\Lo (a))$, and let
\[
\Delta(H)=\bigcup_{a\in H} \Delta(a)
\]
be the set of (successive) distances of $H$. The minimal distance of the
monoid is the minimum of this set.

A monoid is called \emph{half-factorial} if $|\Lo(a)|=1$ for each $a\in H$, i.e., for each element every factorization of this element has the same length 
(yet it might still have several essentially different factorizations).
Equivalently, $H$ is half-factorial if and only if $\Delta(H)=\emptyset $.

\subsection{Abelian groups and zero-sum sequences}

We use additive notation for abelian groups. We denote by $\ord g \in \N \cup \{\infty \}$ the order of an element $g\in G$. Yet note that we almost exclusively deal with finite, or at least torsion, abelian groups, indeed most of the time cyclic ones, so that the order, in this paper, is essentially always finite. We denote by $\mathsf{r}(G)$ and $\mathsf{r}^{\ast}(G)$ the rank and the total rank, respectively, of an abelian group $G$; for a finite cyclic group with at least two elements -- the only case of actual relevance here -- this is of course $1$ and $\omega(|G|)$, the number of distinct prime divisors of the order of $G$, respectively.

Let $G$ be an abelian group, and $G_0 \subset G$.
Let $S\in \fc(G_0)$ be a sequence over $G_0$.
The notions of length and multiplicity were already mentioned. We recall
some more specific notions.
We call
\begin{itemize}
\item $\supp(S)= \{g \in G_0 \colon \vo_g(S) > 0\}$, the \emph{support} of $S$,
\item $\s(S)=\sum_{g\in G_0} \vo_g(S)g \in G$, the \emph{sum} of $S$,
\item $\Sigma(S)= \{\s(T) \colon 1 \neq T \mid S\}$, the set of \emph{subsequence sums} of $S$,
\item and assuming all elements appearing in $S$ have finite orders, we let
\[
\ko(S)=\sum_{g\in G_0} \frac{\vo_g(S)}{\ord g} \in \Q
\]
be the \emph{cross number} of $S$.
\end{itemize}

Moreover, if $G_0 \subset \langle e \rangle$ (the group generated by
$e$) and $e$ has finite order, then let
$\s_e$ denote the homomorphism from  $\fc(G_0)$ to $\N_0$ defined, for
$h \in G_0$, by $\s_e(h)= a$ where $a\in [1, \ord e]$ such that
$h= ae$.
Note that $\s(S)= \s_e(S)e$, in particular for each $B$ with $\s(B)=0$
one has
$\ord e \mid \s_e(B)$. Moreover, note that for $g$ another generating element of
$\langle e \rangle$ we have 
\begin{equation}
\label{eq_psigma}
\s_e(h) \equiv \s_e(g)\s_g(h) \pmod{\ord e}.
\end{equation}
The notation we use is slightly non-standard. In case an element $e$ is a priori fixed, the notation $\s_0$ was used for $\s_e$, however this is not sufficiently general for our purpose. Another notation that is used is $||\cdot ||_e$, called the $e$-norm of a sequence. However, we feel it has a slightly different connotation, and in our context it is also slightly less convenient on typographical grounds.

A sequence $S\in \fc(G_0)$ is  called a \emph{zero-sum sequence} if $\s(S)=
0$ and it is called \emph{zero-sum free} if $0\notin \Sigma(S)$. A zero-sum
sequence is called a \emph{minimal zero-sum sequence} if it is non-empty and
has no proper zero-sum subsequence.
Let
\[
\bc(G_0)= \{S\in \fc(G_0) \colon \s(S) = 0\}
\]
denote the set of zero-sum sequences over $G_0$; evidently, $\bc(G_0)$
is a submonoid of $\fc(G_0)$. And, let $\ac(G_0)$ and
$\ac^{\ast}(G_0)$ denote the set
of minimal zero-sum sequences and zero-sum free sequences over $G_0$,
respectively; note that $\ac(G_0)$ is the set of irreducible elements
of $\bc(G_0)$.

If $G_0$ consists of torsion elements only, we denote by
\[
\Ko(G_0)=\sup\{\ko(A)\colon A\in \ac(G_0)\}
\]
the \emph{cross number} of $G_0$ and by
\[
\ko(G_0)=\sup\{\ko(A)\colon A\in \ac^{\ast}(G_0)\}
\]
the \emph{little cross number} of $G$.
We recall (see for example \cite[Theorem 5.5.5]{geroldingerhalterkochBOOK}) that, for $G$ a finite abelian group, we have
the following two inequalities:
\begin{equation}
\label{star33}
\ko(G) \le \log |G|\quad \text{and} \quad \Ko(G)\le \frac{1}{q} +
\log|G|,
\end{equation}
where $q$ is the smallest prime divisor of $|G|$.

\subsection{Krull and block monoids}
\label{ssec_km}

A monoid $H$ is called a \emph{Krull monoid} if there exists a free monoid $\fc(P)$
and a homomorphism $\varphi : H \to \fc(P)$ such that
\[a \mid_H b \text{ if and only of  } \varphi(a) \mid_{\fc(P)} \varphi(b)\]
for all $a,b \in H$; such a homomorphism is called a divisor
homomorphism (into a free monoid).
If in addition, for each $p\in P$ there exist $a_1, \dots, a_k \in H$
such that $\gcd\{\varphi(a_i)  \colon i\in [1,k]\}=p$, then $\fc(P)$
is called a monoid of divisors of $H$ and $\varphi$ a divisor
theory. Every Krull monoid has an essentially unique divisor theory.
Recall that there are various equivalent definitions for Krull monoids,
e.g., as completely integrally closed and $v$-noetherian monoids; for
a detailed account see \cite{halterkoch98book}, especially Chapters 22 and 23.

Let $H$ be a Krull monoid with divisor theory $\varphi: H \to \fc(P)$.
Then $\cc(H)= \qo(\fc(P))/ \qo(\varphi(H))$ is called the class group of $H$;
up to isomorphism a Krull monoid has a unique divisor theory, thus
$\cc(H)$ does not depend on $\varphi$.
Moreover, we denote by $\dc(H) = \{g \in \cc(H) \colon g \cap P \neq \emptyset\}$ the set of classes containing prime
divisors.

Consider
\[
\tilde{\beta} 
\begin{cases}
\fc(P)& \to \fc(\dc(H)) \\
 p_1  \dots p_{\ell}  	& \mapsto [p_1]\dots [p_{\ell}]
\end{cases}
\]
where $[p_i]$ denotes the class containing $p_i$.
Note that the image of $\tilde{\beta} \circ \varphi$ is equal to $\bc(\dc(H))$.
The map $\beta: H \to \bc(\dc(H))$, induced by $\tilde{\beta}$,
is called the block homomorphism, and $\bc(\dc(H))$ is called the
block monoid associated to $H$. The block homomorphism is a transfer
homomorphism (indeed, the archetypal example of a transfer
homomorphism).

In particular, we have
\[\Lo_H(a) = \Lo_{\bc(\dc(H))}(\beta(a))\]
for each $a\in H$ and
\[
\Delta(H) = \Delta(\bc(\dc(H))).
\]

We point out that the block monoids themselves and more generally  monoids of zero-sum sequences are Krull monoids, as the embedding $\bc(G_0)\hookrightarrow \fc(G_0)$ is easily seen to be a divisor
homomorphism (yet, in general, this is not a divisor theory).

We collect relevant examples of structures to which our results apply, that is structures that are Krull monoids or structures that admit a transfer homomorphism to a Krull monoid with potentially finite cyclic class group. 

To start we note that a domain is a Krull domain if and only if its multiplicative monoid is a Krull monoid, by a result due to Krause \cite{krause89}. Furthermore, Dedekind domains and more generally integrally closed noetherian domains are Krull domains (see, e.g., \cite[Section 2.11]{geroldingerhalterkochBOOK}). 
The notions of class group recalled above for monoids coincide with the usual ones in those context; that is, it is the ideal class group for a Dedekind domain and the divisor class group for a Krull domain. 
Indeed,  for a Dedekind domain $D$, if $\ic^{\bullet} (D)$
denotes its non-zero ideals (a free monoid over the non-zero prime
ideals), the map
\[\varphi: \begin{cases}
D\setminus \{0\} & \to \ic^{\bullet} (D) \\
 a & \mapsto aD \end{cases}\]
is a divisor theory; recall that the monoid of ideals is a free monoid generated by the prime ideals.
For Krull domains and monoids, it suffices to consider (non-zero) divisorial ideals, also
called $v$-ideals, instead (with $v$-multiplication as operation).

We now give some more specific examples of Krull monoids and domains.  
\begin{itemize}
\item Rings of integers in  algebraic number fields and more generally holomorphy rings in global fields (see, e.g., \cite{geroldingerhalterkochBOOK}, in particular Sections 2.11 and 8.9). 
\item Regular congruence monoids in Dedekind domains, for example the domains mentioned above (see, e.g., \cite{geroldingerhalterkoch04cong} or \cite[Section 2.11]{geroldingerhalterkochBOOK}).  
\item Rings of polynomial invariants of finite groups (see, e.g., \cite[Theorem 4.1]{cziszteretal}.
\item Diophantine monoids (see, e.g., \cite{chapmankrauseoek02}, especially Theorem 1.3 for examples with finite cyclic class group).
\end{itemize}

Another source of examples are semi-groups of isomorphy classes of certain modules (the operation being the direct sum). These are Krull monoids in various cases. For an overview we refer to the monograph of Leuschke and Wiegand \cite{leuschkewiegand}.  We highlight a result of particular relevance to us by Baeth and Geroldinger \cite[Theorem 5.5]{baethgeroldinger14}, yielding  Krull monoids with finite cyclic class group (of any order) such that each class contains a prime divisor (earlier example often had infinite class groups).

Besides Krull monoids themselves there are other structures that are not Krull monoids themselves, for example they might be not commutative or not integrally closed, yet still admit a transfer homomorphism to a Krull monoid.
For such a structure the  system of sets of lengths is still equal to that of a Krull monoid, and our results apply (in case of finite cyclic class group).

The first example below is due to Smertnig \cite[Theorem 1.1]{smertnig13}, the second due to Geroldinger, Kainrath, and Reinhart \cite[Theoorem 5.8]{geroldingerkr} (their actual result is more general); the first is not commutative, the second not integrally closed.   

\begin{itemize}
\item 
Let $\mathcal{O}$ be a holomorphy ring in a global field  and let $A$ be a central simple algebra over this field. For $H$ a classical maximal $\mathcal{O}$-order of $A$ one has that if  every stably free left $H$-ideal is free, then there is a transfer-homomorphism from $H\setminus \{0\}$ to the monoid of zero-sum sequence over  a ray class group of $\mathcal{O}$, which is a finite abelian group.

\item Let $H$ be a seminormal order in a holomorphy ring of a global field with principal order $\widehat{H}$ such that the natural map $\mathfrak{X} (\widehat{H}) \to \mathfrak{X} (H)$ is bijective and there is an isomorphism  between the $v$-class groups of $H$ and $\widehat{H}$. Then there is a transfer-homomorphism from $H\setminus \{0\}$ to the monoid of zero-sum sequence over this $v$-class group, which is a finite abelian group.
\end{itemize}

\subsection{Key notions}
\label{ssec_kn}

Having all the preceding preparatory notions at hand, we collect
those notions that are most relevant to the present paper
(some of them were already informally discussed).

Let $H$ be an atomic monoid. The \emph{set of minimal distances} of $H$ is defined as
\[
\Dast{H} = \{ \min \Delta(S) \colon S \subset H \text{ divisor closed
}, \, \Delta(S) \neq \emptyset \}.
\]

It is traditional to exclude half-factorial submonoids $S$ in the
definition, so that $0$ is never an element of $\Dast{H}$, as opposed
to always;
note that we use the convention that $\min \emptyset = \gcd \emptyset = 0$.

It is well known (see for example \cite[Proposition 1.4.4]{geroldingerhalterkochBOOK}) that indeed
\begin{equation}
\label{eq_min=gcd}
\min \Delta(H) = \gcd \Delta(H).
\end{equation}
In particular, if $S \subset H$ is a divisor-closed submonoid, then
$\Delta(S) \subset  \Delta(H)$ and thus $\min \Delta(H) \mid \min \Delta(S)$.

For $G$ an abelian group (in this paper essentially only finite
cyclic groups occur) we write $\Dast{G}$ instead of $\Dast{\bc(G)}$
and we note that $S\subset \bc(G)$ is divisor closed if and only if
$S= \bc(G_0)$ for some $G_0 \subset G$.

An atomic monoid $H$ is called \emph{$d$-congruence half-factorial}, or \emph{congruence half-factorial of order $d$}, if for
each $a\in H$ one has $\ell \equiv \ell' \pmod{d}$ for all $\ell ,
\ell' \in \Lo(a)$.
It follows from \eqref{eq_min=gcd}
that $H$ is $d$-congruence half-factorial if and only
if $d \mid \min \Delta(H)$.

We will study the set $\Dast{G}$ and for $d\in \Dast{G}$ the structure
of those sets $G_0 \subset G$ such that $d= \min \Delta(G_0)$.

We already mentioned the relevance of $\Dast{G}$ for describing sets of lengths.We recall this connection and relevant notions in detail.

We recall the Structure Theorem for Sets of Lengths for Krull monoids with finite class group. It is a central result in factorization theory; the initial version is due to Geroldinger \cite{geroldinger88}, the refined version we recall below, is due to Freiman and Geroldinger  \cite{freimangeroldinger00}.
Moreover, we add that a result of this form is also known for other classes of monoids, in particular certain though not all Krull monoids with infinite class group (see \cite[Chapter 4]{geroldingerhalterkochBOOK} for an overview, and \cite{geroldingergrynkiewicz09}, \cite{geroldingerkainrath10} for recent contributions).

To state the Structure Theorem, we first recall a definition.

\begin{definition}
A finite set of integers $L$ is called an almost arithmetical
multiprogression (AAMP) with difference $d$, period
$\{0, d \} \subset \dc  \subset [0,d]$ and bound $M$
if there exists some integer $y$ such that
\[
L = y + (L' \cup L^{\ast} \cup L'') \subset y + (\dc + d \Z)
\]
with 
\[ 
L^{\ast} =   (\dc + d \Z) \cap [0, \max L^{\ast}], \quad L' \subset [-M,-1]
\] 
and
\[
L'' \subset [\max L^{\ast}+1,\max L^{\ast}+ M].
\]
\end{definition}

Now, we recall the Structure Theorem for Sets of Lengths.

\begin{theorem}
\label{thm_STSL}
Let $H$ be a Krull monoid with finite class group $G$ satisfying $|G|\ge 3$.
There exists some $M\in \mathbb{N}_0$ such that for each $a\in H$,
$\Lo(a)$ is an AAMP with difference $d$ in $\Dast{G}$ and bound $M$.
\end{theorem}

It is crucial that the bound $M$ and the set $\Dast{G}$ are finite
and independent of $a$; otherwise the statement would be trivial.
Moreover, it was recently proved, though there was some evidence
for this before, that this structural description is in a certain
sense optimal (see \cite{WASreal}). Note that, for  $|G|\le 2$,
the monoid is half-factorial so that the condition $|G|\ge 3$ merely excludes corner-cases.

Of course, such a result would also hold for suitable sets other than $\Dast{G}$, such as any superset of it, yet $\Dast{G}$ is the natural choice; in particular if one wishes a condition that depends on the class group only, which is typically the case, or is in a situation where every class contains a prime divisor.
Thus, to gain information on $\Dast{G}$ is key towards a more precise
understanding of sets of lengths and thus an important problem
of  factorization theory.

From its early beginning on it is classical in factorization theory to consider quantitative problems, too
(see, e.g., Narkiewicz' monograph \cite[Chapter 9]{narkiewicz04}).
That is, one is for example interested in the number, in an asymptotic
sense and up to associates, of algebraic integers (of some number field)
that have a certain factorization property; or equivalently, the number of principal ideals with the respective factorization property.
Indeed, one can ask this question for other structures, such as elements
of a holomorphy ring of an algebraic function fields over a finite field,
or consider the problem in a suitable abstract setting.

We do not discuss this in detail here and only consider the classical setting of rings of algebraic integers; the results we establish apply verbatim in the other or more general context. It should also be noted that more precise asymptotic results than recalled below can be obtained, in particular an asymptotic  equality instead of merely the order, but this is not relevant here. For details on the results touched upon below, we refer to \cite[Section 9.4]{geroldingerhalterkochBOOK}.

Let $K$ be an algebraic number field, let $\oc_K$ denote its ring of integers, and let $H$ denote the non-zero principal ideals of $\oc_K$. Note that $H$ is a Krull monoid with the same class group as the ideal class group of $\oc_K$.
In view of the Structure Theorem for Sets of Lengths the following definition is a natural one:

Let $d\in \N$, $M\in \N_0$ be sufficiently large (this could be made explicit), and $\{0,d\}\subset \dc \subset [0,d]$.
Then $\pc(H,\dc,M)$ denotes the set of all $a\in H$ such that $\Lo(a)$
is an AAMP with period $\dc$ and bound $M$, and such that
\[
\max\Lo(a) - \min \Lo(a)\ge 3M + (\max \Delta(H))^2.
\]
These conditions are in place to guarantee that it makes sense -- to the extent possible -- to say that $\dc$ is `the' period of $\Lo(a)$. On the one hand, to give an extreme example, a singleton is an AAMP with period $\dc$ for every $\dc$, so one needs some condition. On the other hand, a set $L$ that is an AAMP with period $\dc$ (and difference $d$) and some bound $M$ is also an AAMP with period
\[
\bigcup_{i=0}^{k-1 } (id +  \dc)
\]
(and difference $kd$) and bound $M$; moreover, it is typically possible to slightly `shift' the central part so that $L$ 
is also an AAMP with period $\dc'$ where $\overline{\dc'}$ equals $\overline{s+\dc}$ for some $s\in [0,d]$ 
and the bar denotes the projection onto $\Z/d \Z$.
The conditions guarantee that this mild level of non-uniqueness, just illustrated, is the only one.

Let $\pc(H, \dc, M)(x)$ denote the counting function associated to this set,
that is the number of elements in $\pc(H, \dc, M)$ of (absolute) norm at most $x$.
It is known  that if $\pc(H, \dc, M)\neq \emptyset$, then
\[
\pc(H, \dc, M)(x) \asymp \frac{x}{(\log x)^{1 - \ao/|G|}} (\log\log x)^{\bo}
\]
where $\ao$ and $\bo$ 
are non-negative integers and $G$ denotes
the class group of $H$.
For the precise description of $\ao$ and $\bo$,
see Section \ref{sec_nt};
here we only mention that both depend only on $\dc$ and the class group, and $\ao$ 
is intimately linked to the inverse problem -- that is, the strucuture of sets $G_0$ with specified $\mD{G_0}$ -- associated to elements of $\Dast{G}$ that are related to $\max \dc$, mainly its multiples but depending on $\dc$ possibly also (certain) divisor.

\section{Some auxiliary results}
\label{sec_aux}

For sets of the form $G_0=\{e,ae\}$ (with $a\in [1, \ord e]$) the problem of determining $\min \Delta (G_0)$ is solved, by the work of Chang, Chapman, and Smith \cite{changetal07} building on work by
Geroldinger \cite{geroldinger90b}.
First, we recall the special case mainly relevant for our purpose in a form convenient for the present application; then we comment briefly on the context.

\begin{theorem}
\label{aux_thm_2elements}
Let $G$ be a finite cyclic group, $e$ be a generating element of $G$ and $a\in [1,|G|]$ such that $\gcd(a,|G|)=1$.
Then $\min \Delta(\{ e, ae \}) > \sqrt{|G|}$ if and only if there exist
some positive integers $c_1$ and $c_2$ such that
\[
a = \frac{|G| - c_1}{c_2}
\]
and the quantity
\[
d_a = \frac{|G| - (c_1 + c_2)}{c_1 c_2}
\]
is integral and satisfies $d_a> \sqrt{|G|}$. Indeed, in this case $\min \Delta( \{ e , ae \})=d_a$.
\end{theorem}

By the above-mentioned works it is known that $\min \Delta(\{e, ae\})$, using the notation of the above result, can be expressed in terms of the continued fraction expansion of $n/a$. The fact that continued fractions play a role in this context can be roughly understood by noting that a minimal zero-sum sequence containing $ae$ with multiplicity $k$ exists if and only if $\lceil ka/n \rceil n  -  ka$ is smaller than $\lceil ja/n \rceil n  -  ja$ for each $0<j<k$, and observing the connection to `good' rational approximations of $n/a$.

More specifically, Corollary 3.2 of \cite{changetal07} asserts that if $\min \Delta(\{e, ae\})$ is greater than $\sqrt{|G|}$, then the odd continued fraction expansion of $n/a$ is of length $3$; 
the `odd' means that the lengths of the continued fraction expansion is odd, which can be achieved by (in turn) allowing that the last term is $1$. We point out that \cite[Corollary 3.2]{changetal07} is formulated for $n$ prime only, as is \cite[Lemma 3.1]{changetal07} on which it is based, yet its proof carries over verbatim, and the only result that is used, i.e. \cite[Theorem 2.1]{changetal07}, is formulated for general $n$. In fact, up to here the argument even works for $a$ not coprime to $n$ with a minimal modification of \cite[Lemma 3.1]{changetal07}.
Moreover, it is asserted in \cite{changetal07} (see specifically equation (1), Proposition 3.4, and the subsequent discussion in that paper) that the odd continued fraction expansion of $n/a$ is of length $3$, as there let us denote it by $[b,d,c]$, if and only if 
\[
a = \frac{n - c}{b}\quad  \text{ and } \quad  d= \frac{n-b-c}{bc}.
\]
Moreover, in this case, $d = \min \Delta(\{e, ae\})$. Again, the original discussion is for $n$ prime, yet the condition that $a$ and $n$ are coprime suffices for the arguments to hold.
These last conditions are precisely what is encoded in the conditions of the above result. Also, note that if $\frac{n-c}{b}$ and $\frac{n-b-c}{bc}$ are positive integers, then $\frac{n-c}{b}$ is invertible 
modulo $n$ and its inverse is $\frac{n-b}{c}$.
We point out that the condition that the minimal distance is large is only needed to guarantee that the continued fraction expansion is of length $3$. The remainder of the argument does not need this size-condition. For later reference we formulate this as a remark. 

\begin{remark}
Let $n$ be a positive integer and let $G$ be a finite cyclic group of order $n$. Let $e$ be a generating element of $G$ and let  $b,c\in [1,n]$ such that $\frac{n-c}{b}$ and $\frac{n-b-c}{bc}$ are positive integers. Then 
\[
\min\Delta \left( \left\{ e, \frac{n-b}{c}e\right\} \right) = \frac{n-b-c}{bc}.
\] 
\end{remark}

Finally, observe that while with more work and at the expense of a more complicated formulation the restriction that $a$ and $|G|$ are coprime could be avoided, this condition is essentially irrelevant in view of Lemma \ref{aux_lem_higherorder} below.

The following lemma of Geroldinger \cite{geroldinger90b} (also see
\cite[Lemma 6.8.5]{geroldingerhalterkochBOOK}) gives a simplified way of
determining $\min \Delta(G_0)$ for $G_0$ a subset of a finite cyclic group
containing a generating element. We use it frequently and it was also used in the proof of the above-mentioned result.

\begin{lemma}
\label{aux_lem_geroldinger}
Let $G$ be a finite cyclic group and let $e$ be a generating element of $G$.  
Further, let $G_0$ be a subset of $G$ containing $e$. Then,
\[
\mD{G_0}= \gcd \left\{\frac{\s_e(A) - |G| }{|G|}\colon A\in
  \ac(G_0)\right\}. 
\]
\end{lemma}

If $G_0$ contains more than one generating element one is evidently
free to choose any of these generating elements as the distinguished
one; yet, it needs to be fixed throughout an argument.

Based on this lemma, we obtain the following result.
\begin{lemma}
\label{smallelements}
Let $G$ be a finite cyclic group and let $e$ be a generating element of $G$. Further, let $G_0$  
be a subset of $G$ containing $e$. Let $x\in [1,|G|]$ such that  $\min \Delta(G_0 \cup \{xe\})\geq x$.
Then,
\[
\min \Delta(G_0\cup \{xe\}) =  \min \Delta(G_0).
\]
\end{lemma}

\begin{proof}
We set $n=|G|$. For $B\in \mathcal{B}(G_0\cup \{xe\})$ with $(xe)\mid B$, let
\[
f(B)=
\begin{cases}
e^{x}(xe)^{-1}B & \text{ if } (xe)\mid B,\\
B              & \text{ otherwise.}
\end{cases}
.
\]
We note that $f(\sigma(B))=\sigma(f(B))$, and thus  $f(B)$ is a zero-sum sequence, too. First, suppose that $f(\mathcal{A}(G_0 \cup \{xe\}))\subset
\mathcal{A}(G_0 \cup \{xe\})$; in other words, the minimality of
zero-sum sequences is preserved under this replacement.
We claim that in this case we have 
\[
\min \Delta(G_0\cup \{xe\}) = \min \Delta(G_0).
\]
To see this, it suffices to note that $\s_{e}(A)=
\s_e(f(A))$ and that
$f^n(A)\in \ac(G_0)$, as $\vo_{xe}(A) \le n$, which together with
Lemma \ref{aux_lem_geroldinger} implies that
$\mD{G_0}= \mD{G_0 \cup \{xe\}}$.

So, we may now assume that there exists some
$A\in \mathcal{A}(G_0 \cup \{xe\})$ such that $f(A)$ is not a
\emph{minimal} zero-sum sequence. We consider
\[
C=e^nA=(e^{n-x}(xe))f(A).
\]

By definition of $C$, we have
\begin{equation}
\label{2LC}
2\in  \mathsf{L}(C).
\end{equation}
We also have $1+ \mathsf{L}(f(A))\subset \mathsf{L}(C)$ which
implies $1 + \min \mathsf{L}(f(A)) \in \mathsf{L}(C)$. But,
since $f(A)$ is not a minimal zero-sum sequence by assumption
(and clearly non-empty), we have $\min \mathsf{L}(f(A))\ge 2$.
It follows that $1 +  \min \mathsf{L}(f(A)) \ge 3$. With \eqref{2LC},
this implies
\[
\min \mathsf{L}(f(A))-1 = \big( 1 +  \min \mathsf{L}(f(A)) \big) - 2 \ge \min \Delta (C).
\]
Using $\min \mathsf{L}(f(A))\le x$, as otherwise we would
get a non-trivial zero-sum subsequence of $A$, we then infer
\[
\min \Delta(C) \le \min \mathsf{L}(f(A))-1 \le x -1,
\]
implying the claim.
\end{proof}

The following lemma is useful to reduce the complexity of certain
arguments.
It can be seen as part of the framework of higher-order block monoids as introduced in \cite{WAShigherorder}; though for this
special case one could avoid that machinery.
For the first result see in particular \cite[Proposition 4.1]{WAShigherorder}; while the latter is also easily derivable from results of that paper, it is in fact already contained in \cite[Theorem 6.7.11]{geroldingerhalterkochBOOK} (inspecting the proof, one sees that $G_0$ being finite there is not relevant).

\begin{lemma}
\label{aux_lem_higherorder}
Let $G$ be an abelian torsion group, and let $G_0\subset G$.
For $g \in G_0$, let $\mathsf{n}_{G_0}(g)$ denote the smallest $n\in
\N$ such that $n g \in \langle G_0 \setminus \{g\} \rangle$.
\begin{enumerate}
\item Let $G_1= \{\mathsf{n}_{G_0}(g)g\colon g \in G_0\}$. The map
  $\bc(G_0)\to \bc(G_1)$ induced by $h^{\mathsf{n}_{G_0}(h)} \mapsto
  \mathsf{n}_{G_0}(h) h$ for each $h \in G_0$ is well-defined and a
  transfer homomorphism. In particular, $\lc(G_0)=\lc(G_1)$.
\item Let $g \in G_0$ and
$G_g= (G_0 \setminus \{g\})\cup \{\mathsf{n}_{G_0}(g) g \}$.
The map $\bc(G_0)\to \bc(G_g)$ induced by $g^{\mathsf{n}_{G_0}(g)}
\mapsto \mathsf{n}_{G_0}(g) g$ and $h \mapsto h$ for each other $h$ is
well-defined and a transfer homomorphism. In particular, $\lc(G_0)=\lc(G_g)$.
\end{enumerate}
\end{lemma}

By `well-defined' we mean that $\bc(G_0)$ is contained in the
(free) submonoid of $\fc(G_0)$ generated by  $\{h^{\mathsf{n}_{G_0}(h)}\colon h \in G_0\}$ and $\{g^{\mathsf{n}_{G_0}(g)}\}\cup (G_0\setminus \{g\})$, respectively; thus, the above-mentioned maps clearly make sense.

As we use it occasionally we make the following remark, which is a direct consequence of this lemma, and the results we recalled after Theorem \ref{aux_thm_2elements} or also Lemma \ref{lem_charhf}.

\begin{remark}
\label{rem_2el}
For $G$ finite cyclic and $G_0 \subset G$ with $|G_0|=2$,
we have that $\Delta(G_0)=\emptyset$ if and only if
$|\{\mathsf{n}_{G_0}(g)g\colon g \in G_0\}|=1$.
In particular, a subset of a finite cyclic group containing two
distinct elements of the same order is never half-factorial.
\end{remark}

The following lemma is relevant in Section \ref{sec_appCHF}.
We do not know whether an analogue holds for finite abelian groups in general.

\begin{lemma}
\label{aux_generating}
Let $G$ be a finite cyclic group.
For each subset $G_0 \subset G$, there exists a generating subset $G_0'\subset G$ such that there is a transfer homomorphism
\[\theta : \bc(G_0')\to \bc(G_0).\]
\end{lemma}

\begin{proof}
Let $n=|G|$.
For $n=1$ this is trivial and we assume $n>1$.
We proceed by induction on $t$ the number of prime divisors counted with multiplicity of $n/ |\langle G_0 \rangle|$.

For $t=0$, of course $G_0$ is generating and we can simply set $G_0' = G_0$.

Suppose $t>0$. Let $p$ be one of the prime divisors. Now, let $g\in G_0$ such that the $p$-adic valuation of $\ord g$ is maximal among all elements of $G_0$. Since $p$ divides $n/ |\langle G_0 \rangle|$ the $p$-adic valuation of $\ord g$ is less than the of $n$. Therefore, there exists some $h\in G\setminus \{0\}$ such that $ph=g$.   
We observe that  $\ord h =p \ord g$ (note that since we assumed $h \neq 0$ this must even hold true if $g=0$), and we infer that $h \notin G_0$.
Moreover, it is easy to see that $p$ is the minimal positive integer $j$ such that $j h \in \langle G_0 \rangle$. Thus, Lemma \ref{aux_lem_higherorder} implies that there exists a transfer homomorphism $\theta' : \bc(G_0 \cup \{h\}) \to \bc(G_0)$. Moreover, since $\langle G_0 \rangle $ is a proper subset of $ \langle G_0 \cup \{h\} \rangle $, we get by the induction hypothesis that there exists a generating subset $G_0'$ of $G$ such that there is a transfer homomorphism $\theta'' : \bc(G_0')\to \bc(G_0 \cup \{h\})$. Since the composition of transfer homomorphisms is again a transfer homomorphism, setting $\theta = \theta' \circ \theta''$, the claim is proved.
\end{proof}

\section{Non-large-cross-number sets and related notions}
\label{non-LCN}

The main purpose of this section is to prove Theorem \ref{prop_2subset}.
To this end we need some additional notions namely that of weakly half-factorial sets and large-cross-number sets, which we recall below.
We then establish some results involving these notions that are used to prove the above-mentioned proposition. These results are proved in more generality than needed, 
since we believe they are of some independent interest.

The (arithmetic) definition of a half-factorial subset was already recalled in Section \ref{sec_prel}. We recall the classical characterization of half-factorial sets of abelian torsion groups; it goes back to -- independently yet with minor variations on the generality of the result  --  Skula \cite{skula76}, {\'S}liwa \cite{sliwa76}, and Zaks \cite{zaks76}; for a modern proof see \cite[Proposition 6.7.3]{geroldingerhalterkochBOOK}.

\begin{lemma}
\label{lem_charhf}
Let $G$ be an abelian torsion group, and let $G_0$ be a subset of $G$.
The set $G_0$ is half-factorial if and only if $\ko(A)=1$ for each $\ac(G_0)$.
\end{lemma}

The following two notions are inspired by this characterization; the former very directly, the latter more implicitly.
The notion of a weakly half-factorial set was introduced by
{\'S}liwa \cite{sliwa82}, using a different terminology;
for more recent investigations on these sets, see \cite{WAS13}
and also \cite[Section 6.7]{geroldingerhalterkochBOOK}.
The terminology large-cross-number set was introduced in
\cite{WASchar}, however the underlying idea is older; see,
e.g., the work of Gao and Geroldinger \cite{gaogeroldinger00}.

\begin{definition}
\label{defLCNWHFHF}
Let $G$ be an abelian torsion group, and let $G_0$ be a subset of $G$.
\begin{enumerate}
\item We say that $G_0$ is a \emph{large-cross-number set}  if $\ko(A)\ge 1$ for each $A\in \ac(G_0)$.
\item We say that $G_0$ is a \emph{weakly half-factorial set} if $\ko(A)\in \N$ for each $\ac(G_0)$.
\end{enumerate}
\end{definition}

From the characterization in Lemma \ref{lem_charhf} and the definitions, it follows that each half-factorial set is a weakly half-factorial set, and
each weakly half-factorial set is a large-cross-number set.

Weakly half-factorial sets in finite cyclic groups have a simple structure; namely,
we have the following result at our disposal
(see \cite{sliwa82}, and for more general results,
using the terminology of the present paper, see \cite{WAS13}
or \cite[Theorem 6.7.5]{geroldingerhalterkochBOOK}).

\begin{lemma}
\label{lem_struct_cycl45}
Let $G$ be a finite cyclic group. A set subset $G_0$ of $G$ is weakly half-factorial if and only if
there exists some generating element $e$ of $G$ such that
$G_0 \subset \{de \colon d \mid |G| \}$.
\end{lemma}

We start by showing that a subset of a finite cyclic group  that is not a large-cross-number set 
contains a non-half-factorial subset of cardinality $2$.

\begin{lemma}
\label{lem_2innonLCN}
Let $G$ be a finite cyclic group, and let $G_0$ be a subset of $G$ that is not weakly half-factorial.
There exists a subset $G_2 \subset G_0$ with $|G_2|=2$ that is not weakly half-factorial.
\end{lemma}

\begin{proof} To simplify the writing of the proof, we define $n=|G|$.
Let $e$ be a generating element of $G$ such that $G_0 \cap \{de \colon
d \mid n \}$ has maximal cardinality; we denote this set by $G_0'$ and
set $D_0= \sigma_e(G_0')$, the set of the respective divisors of
$n$. We denote $G_d=\{de \colon d \mid n\}$.

Since $G_0$ is not weakly half-factorial, we know by Lemma \ref{lem_struct_cycl45}
that there exists some $h \in G_0 \setminus G_0'$.
Let $n_h\in \mathbb{N}$ be minimal such that $n_h h \in \langle G_0'\rangle$.

We have that $n_h\mid n$, and more precisely, for $d'=\gcd(\s_e(h),n)$,
we have
\[
n_h = \frac{\lcm(d', \gcd D_0)}{d'}= \frac{\gcd {D_0}}{ \gcd(d' , \gcd D_0)}.
\]
By Lemma \ref{aux_lem_higherorder}, we know that each zero-sum sequence
over $G_0'\cup \{h\}$ contains $h$ with a multiplicity
that is a multiple of $n_h$ (this includes $0$),
and the replacement $h^{n_h}\mapsto n_h h$ induces a surjection from
the minimal zero-sum sequences over $G_0'\cup \{h\}$ to the ones over
$G_0' \cup \{n_h h\}$, which also preserves the cross number of zero-sum sequences.

Thus, if $n_h h \in G_d$, then  $G_0' \cup \{n_h h\}$ is weakly half-factorial.
And, as the cross numbers of minimal zero-sum sequences are the same,
we get that $G_0' \cup \{h\}$ is weakly half-factorial. Yet, this contradicts the choice of $e$,
since if $G_0'\cup \{h\}$ is weakly half-factorial, then there exists some $e'$ such that
$G_0'\cup \{h\} \subset \{d e' \colon d \mid n\}$, contradicting the choice of $e$.

So, we may assume that $n_hh \notin G_d$.
This means that
\[n_h\s_e(h) \not \equiv d \pmod{n}\]
for each $d\mid n$.

Let $m\in D_0$, and let $n_m\in \mathbb{N}$ minimal such that
$n_mh \in \langle m e\rangle$; that is,
\[
n_m = \frac{\lcm(m,d')}{d'}.
\]
We have that $\lcm(m,d')\mid n_m \s_e(h)$. Let $k_m\in [1,n]$ be congruent to
$n_m \s_e(h)$ modulo $n$; we also have $\lcm(m,d')\mid k_m$.
We consider the sequence
\[
(me)^{\frac{n-k_m}{m}}\quad  h^{n_m}.
\]
This is a minimal zero-sum sequence and its cross number is
\[
\left( \frac{1}{n/m} \right) \frac{n-k_m}{m} + \frac{n_m}{\ord h}=
\frac{n-k_m}{n} +  \frac{\lcm(m,d')}{n}
= 1 + \frac{\lcm(m,d') - k_m}{n}.\]
Now, if this is not integral then the set $\{me,h\}$ is not weakly half-factorial,
and we are done. Thus, suppose this is integral and thus at least $1$,
and so $k_m = \lcm(m,d')$. Consequently,
\begin{equation}
\label{eq_2innonLCN}
n_m \s_e(h) \equiv \lcm(m,d') \pmod{n},
\end{equation}
and inserting the explicit expression for $n_m$ we have
\[\frac{\s_e(h)\lcm(m,d')}{d'} \equiv \lcm(m,d') \pmod{n} \; ,\]
implying
\[
\s_e(h) \equiv d' \;  \left ( \text{mod} \   \frac{n \gcd(m,d')}{m}  \right ).
\]
Since these congruences hold for each $m \in D_0$, it follows that
\[
\s_e(h) \equiv d' \;
 \left (   \text{mod} \  \lcm\left\{ \frac{n }{m/\gcd(m,d')}\colon m \in D_0 \right\}  \right ) \; ,\]
which we reformulate as
\[ \s_e(h) \equiv  d' \;  \left( \text{mod} \ \frac{n}{ \gcd D_0 /\gcd(\gcd D_0 ,d')} \right ),
\]
implying that
\[ \s_e(h) \frac{ \gcd D_0 }{\gcd(\gcd D_0 ,d')}
\equiv  d'\frac{\gcd D_0}{\gcd(\gcd D_0 ,d')} \pmod{n}.
\]
Yet,
\[
\s_e(h) \frac{\gcd D_0}{ \gcd(\gcd D_0 ,d')}  = \s_{e}(h)n_h
\]
and $d' \gcd D_0 /\gcd(\gcd D_0 ,d')$ is a divisor of $n$.
Thus, we deduce that $n_h h \in G_d$, which contradicts our assumption.
\end{proof}

A slight modification of the proof of this lemma also allows to show
that in finite cyclic groups the notions of weakly half-factorial set and large-cross-number set coincide; and indeed cyclic groups are the only finite abelian groups with that property, and more generally groups of rank $1$ the only torsion abelian groups.
We do not actually apply this result anywhere in the paper; yet,
we include it here as it seems interesting and it can be obtained quickly here.

\begin{proposition}
\label{prop_whflcn}
Let $G$ be an abelian torsion group.
The following two statements are equivalent:
\begin{enumerate}
  \item A subset $G_0 \subset G$ is a weakly half-factorial set if and only if it is a large-cross-number set,
  \item The rank of $G$ is $1$.
\end{enumerate}
\end{proposition}
\begin{proof}
First, suppose the rank of $G$ is $1$, and we have to establish the equivalence of the notions weakly half-factorial and large-cross-number.
Let $G_0\subset G$. It is immediate that a  weakly half-factorial set is a large-cross-number set. We assume that $G_0$ is a larger-cross-number set and  show that it is weakly half-factorial.
Since both properties are of finite character, i.e., they hold for $G_0$ if and only if they hold for every finite subset of $G_0$, we may assume that $G_0$ is finite, or indeed that $G$ is a finite cyclic group.

Now, we proceed as in the proof of Lemma \ref{lem_2innonLCN}.
That is, we pick $e$ a generating element of $G$ such that
$G_0 \cap \{de \colon d \mid n \}$ has maximal cardinality and
denote this set by $G_0'$. We have to show that
$G_0 \setminus G_0' \neq \emptyset$ yields a contradiction.
This is achieved by the exact same argument, except that just
before \eqref{eq_2innonLCN}, we have to infer instead from the fact
that the set is a large-cross-number set that the cross number is at least $1$.

To establish the converse implication, suppose $G$ does not have rank $1$.
It follows that there exists a prime $p$ such that the $p$-rank of $G$ is at least two.
Therefore, it effectively suffices to show that $C_p \oplus C_p$ contains a large-cross-cumber set that is not weakly half-factorial, or in other words we may assume $G= C_p \oplus C_p$ for some prime $p$.
We consider the set $G_1=\{e_1,e_2, -e_1+e_2\}\subset G$ where $e_1,e_2$ are independent.
We observe that
\[\ac(G_1)= \{e_1^j(-e_1+e_2)^je_2^{p-j} \colon j \in [1,p-1]\} \cup \{e_1^p,e_2^p,(-e_1+e_2)^p\}\]
and so
\[\ko(\ac(G_1))= \left \{\frac{j}{p} \colon j \in [p,2p-1] \right\},\]
implying that $G_1$ is a large-cross-number set, yet not weakly half-factorial.
\end{proof}

In a similar vein, we note for the sake of completeness that Lemma \ref{lem_2innonLCN} can be generalized along the same lines.

\begin{proposition}
Let $G$ be an abelian torsion group.
The following two statements are equivalent:
\begin{enumerate}
  \item For each subset $G_0\subset G$ that is not weakly half-factorial, there exists a subset $G_2 \subset G_0$ with $|G_2|=2$ that is not weakly half-factorial,
  \item The rank of $G$ is $1$.
\end{enumerate}
\end{proposition}
\begin{proof}
First, suppose the rank of $G$ is $1$, and $G_0\subset G$ is not weakly half-factorial.
We have to show that $G_0$ contains a subset of cardinality $2$ that is not weakly half-factorial.
It is clear that $G_0$ contains a finite subset $G_0'$ that is not weakly half-factorial, e.g., just consider the support of a minimal zero-sum sequence whose cross number is not integral.
The group generated by $G_0'$ is finite, and thus cyclic. Now, the existence of a not weakly half-factorial set of cardinality $2$ follows by Lemma \ref{lem_2innonLCN}.

Second, suppose $G$ does not have rank $1$.
As the rank of $G$ is by the definition the supremum of the $p$-ranks of $G$ over all primes $p$, it follows that there exists a prime $p$ such that the $p$-rank of $G$ is at least two.
Therefore, it effectively suffices to show that $C_p \oplus C_p$ contains a not weakly half-factorial set not containing such a set of cardinality $2$, or in other words we may assume $G= C_p \oplus C_p$ for some prime $p$.
We consider the set $G_1=\{e_1,e_2, -e_1+e_2\}\subset G$ where $e_1,e_2$ are independent.
We observe that
\[\ac(G_1)= \{e_1^j(-e_1+e_2)^je_2^{p-j} \colon j \in [1,p-1]\} \cup \{e_1^p,e_2^p,(-e_1+e_2)^p\}\]
and so
\[\ko(\ac(G_1))= \left \{\frac{j}{p} \colon j \in [p,2p-1] \right \},\]
implying that $G_1$ is not weakly half-factorial.
Yet, every subset of cardinality $2$ is independent, thus half-factorial and hence also weakly half-factorial.
\end{proof}

To prove Theorem \ref{prop_2subset}, we also use \cite[Proposition 3.6]{WASchar}, which we recall in a simplified form below; on the one hand, we plug-in the upper bound for $\Ko(G)$ recalled in Section \ref{sec_prel} and on the other hand we use the trivial estimate $\ro^{\ast}(G) < 1 + \log |G|$ valid for finite cyclic groups.

\begin{theorem}
\label{thm_boundlcn}
Let $G$ be a finite cyclic group with $|G| \ge 3$, and let $G_0$ be a
non-half-factorial large-cross-number subset of $G$. Then,
\[
\min \Delta(G_0) < \log |G|.
\]
\end{theorem}

We are now ready to prove Theorem \ref{prop_2subset}.

\begin{proof}[Proof of Theorem \ref{prop_2subset}]
We have to show that a subset $G_0$ of a finite cyclic group of order $n>1$ that fulfills $\mD{G_0}\ge \log n$, has a non-half-factorial subset of cardinality $2$.
First, suppose that $G_0$ is weakly half-factorial. It is thus in particular a large-cross-number set. Now, Theorem \ref{thm_boundlcn} yields a contradiction to the assumption on the minimal distance.
Thus, we have that $G_0$ is not weakly half-factorial. In this case, the claim follows by Lemma \ref{lem_2innonLCN}.
\end{proof}

In view of this proof, we see that the condition $\mD{G_0}\geq \log |G|$ stems
directly from Theorem \ref{thm_boundlcn}; in particular, one could
replace the condition by more technical ones, such as the one in \cite[Proposition 3.6]{WASchar} or even an abstract one involving the quantity $\mathsf{m}(G)$, as used in \cite{WASchar}, whose definition we do not wish to recall here.

We point out that for certain types of finite cyclic groups the condition on
$\min \Delta(G_0)\ge \log |G|$ can be dropped altogether, but introducing it is
crucial to obtain a result valid for general finite cyclic
groups. Indeed, Geroldinger \cite[Proposition 6]{geroldinger90b} showed that each
non-half-factorial subset of a cyclic group of prime-power order
contains a non-half-factorial subset of cardinality $2$, and pointed
out that for general finite cyclic groups this assertion is false; the
simplest example is the set $\{e,6e,10e,15e\}$ where $e$ is a
generating element of a cyclic group of order $30$. Moreover, as we
detail below the assertion that every non-half-factorial subset
contains a half-factorial subset of cardinality $2$ is also true for
finite cyclic groups of total rank $2$. This can be summarized in the following corollary, to the proof of Theorem \ref{prop_2subset}.

\begin{corollary}
Let $G$ be a finite cyclic group.
The following statements are equivalent:
\begin{enumerate}
\item Each weakly half-factorial subset of $G$ is a half-factorial set,
\item Each non-half-factorial subset of $G$ contains a non-half-factorial
subset of cardinality two.
\end{enumerate}
\end{corollary}
\begin{proof}
That the former statement implies the latter is clear by the proof of Theorem \ref{prop_2subset}.
To see the converse, it suffices to note that a weakly half-factorial set of cardinality two is half-factorial;
this can be seen using Lemma \ref{lem_struct_cycl45}, and then Lemma \ref{aux_lem_higherorder}; the resulting set has cardinality one.
\end{proof}

To see how this corollary contains the above mentioned results, we recall that if $|G|$ is a prime-power or the product of two
prime-powers, then the condition of the corollary is always, yet non-trivially, fulfilled; however, in case $|G|$ has at least three distinct prime divisors the situation becomes subtle and the pair of assertions of the corollary might or might not  hold (see, e.g., \cite{WAS12} for relevant results).

Regarding the second part of the above proof, we mention the following result from \cite{chapman95}. 
\begin{proposition}
\label{3elementshf}
Let $G$ be a finite cyclic group. Every weakly half-factorial subset of $G$ of cardinality three is half-factorial.
\end{proposition}

\section{Sets containing two elements of maximal order}
\label{sec_spec}

The aim of this section is to investigate $\min \Delta(G_0)$ and
related questions for a subset $G_0$ of a finite cyclic group $G$ with the additional condition that $G_0$ contains at least two generating elements of $G$.
They are main tools in the proofs of our main results that are given in Section \ref{sec_proofs}.

The main result of this section gives a precise description of
the `large' values that $\min \Delta(G_0)$ can attain for these types
of sets; throughout this section the term `large' means  at least $(2 |G|^2)^{1/3}$, that is the bound appearing in Theorem \ref{2maximal}.

In order to make the formulation of our results, at least in part,
somewhat compact, we first introduce some notation. The relevance of the quantities below can be inferred from Theorem \ref{aux_thm_2elements}.

\begin{definition}
\label{def_sets}
Let $n$ be a positive integer.
\begin{enumerate}
\item We let $M (n)$ be the set of triples of positive integers
$(c_1,c_2,d)$ such that $c_1$ and $c_2$ belong to $[1,n]$,
$d\mid n$ and the fractions
\[
\frac{n-c_1-c_2}{c_1c_2} \quad  \text{ and } \quad
\frac{(n-c_1 -c_2)d}{c_2 n }
\]
are positive integers,
\item We let $\In {n}$ denote the set of all positive integers of
the form
\[
\gcd \left\{ \frac{n-c_1-c_2}{c_1c_2}, \,
\frac{(n-c_1 -c_2)d}{c_2 n } \right\} \]
where $( c_1, c_2,d ) \in M(n)$,
\item We let $\Hud{n}$ denote the set of all positive integers which
divide a positive integer of the form
\[ \frac{n-c_1-c_2}{c_1c_2}\]
for some $c_1, c_2$ in $[1,n]$.
\end{enumerate}
\end{definition}

Clearly,  $\In{n} \subset \Hud{n}$. We point out that given a pair of positive integers $(c_1,c_2)$ such that $(n-c_1-c_2)/(c_1c_2)$ is an integer, $(c_1,c_2,d)\in M(n)$ implies that 
\[
\frac{n}{d} \mid (c_1+c_2);
\] 
yet, the converse implication is not true.
Moreover, we recall that $\frac{n-c_1-c_2}{c_1c_2}$ being integral implies that $\frac{n-c_1}{c_2}$ 
is integral and invertible modulo $n$; and its inverse is given by  $\frac{n-c_2}{c_1}$.

Furthermore we point out that for $n \le 13$ we have $(2 n^2)^{1/3} > n-2$; in view of $\max \Dast{G} \le |G| - 2$, as recalled in \eqref{GeroHami} this makes some of the results below trivial for $n \le 13$.

\begin{theorem}
\label{2maximal}
Let $G$ be a finite cyclic group and let $G_0$ be
a subset of $G$ containing at least two generating elements.
Suppose that $\min \Delta(G_0)\ge (2 |G|^2)^{1/3}$.
Then
\[
\min \Delta(G_0) \in \In { |G|}.
\]
Moreover, $G_0$ contains exactly two generating elements, say $e$ and $g$, and they satisfy the conditions
\[
\s_e(g)=\frac{|G|-c_1}{c_2}\quad \text{ and }\quad de=dg
\]
for some $(c_1,c_2, d) \in M(|G|)$ such that
\[
\min \Delta(G_0) =  \gcd \left\{ \frac{|G|-c_1-c_2}{c_1c_2}, \frac{(|G|-c_1 -c_2) d }{c_2 |G|} \right\}.
\]
\end{theorem}

This result is sharp, except for the lower bound on $\min\Delta(G_0)$.
Namely, we obtain the following result, which is essentially a converse to the just mentioned result. 

\begin{proposition}
\label{prop_2maximal}
Let $G$ be a finite cyclic group.
For each $m \in \In{|G|}$ there exists a subset $G_0\subset G$ containing exactly two generating elements such that $\min \Delta(G_0)= m$.
More precisely, for $(c_1, c_2, d) \in M(|G|)$,
and $e$ a generating element of $G$, the element 
\[
g= \left(\frac{|G|-c_1}{c_2}\right)e
\] 
is another generating element of $G$, and
\[
\min \Delta \left( \left \{e, \frac{|G|-c_1}{c_2} e, de \right\} \right) = \gcd \left\{ \frac{|G|-c_1-c_2}{c_1c_2}, \,
\frac{(|G|-c_1 -c_2)d}{c_2 |G|}  \right\}.
\]
\end{proposition}
Since the proof of this result is constructive, we do not need to impose a condition on the size of the minimal distance or other parameters.   

We start with some preparatory results.
First, we show that if a set contains two generating elements,
and a third element that is not of a certain special form,
then the minimal distance of this set, and thus of any set
containing such a set, cannot be `large.'

\begin{proposition}
\label{3elements_improved}
Let $G$ be a finite cyclic group. Let $e,g,h$ be three distinct elements in $G$, such that $e$ and $g$ are generating elements and $h$ is arbitrary (possibly generating).
We set $G_0=\{e,g,h\}$ and $d= |G|/ (\ord h)$.
If $\s_e(h) \neq d$ and $\s_g(h) \neq d$, then
\[
\min \Delta (G_0) < \big( 2|G|^2 \big)^{1/3}.
\]
\end{proposition}

\begin{proof}
We set $n=|G|$ and we assume that $n \ge 13$ as otherwise the result is immediate by the general bound $n-2$ (see \eqref{GeroHami}). We assume that $\s_e(h) \neq d$ and $\s_g(h) \neq d$, and that $\min \Delta (G_0) \ge (2n^2)^{1/3}$.

We start by considering just $\min \Delta (\{e,g\})$.
Since $\{e,g\}$ is a subset of $G_0$ and, by Remark \ref{rem_2el}, is not half-factorial, it follows that $\min \Delta (\{e,g\}) \ge \min \Delta(G_0) \ge (2n^2)^{1/3}$.

Thus, by Theorem \ref{aux_thm_2elements}, we get that
\[\s_e(g)=\frac{n-c_1^{(g)}}{c_2^{(g)}}\]
for some positive integers $c_1^{(g)}, c_2^{(g)}$ such that 
\[
\frac{n- c_1^{(g)} -c_2^{(g)}}{c_1^{(g)}c_2^{(g)}}
\] 
is integral. And, we have
\[
(2n^2)^{1/3} \le \min \Delta (\{e,g\}) =
\frac{n- c_1^{(g)} -c_2^{(g)}}{c_1^{(g)}c_2^{(g)}} <
\frac{n}{c_1^{(g)} c_2^{(g)}}.
\]
It thus follows that both $c_1^{(g)}$ and $c_2^{(g)}$ -- indeed even their product -- are smaller than $(n/2)^{1/3}$.

Next, we consider $\min \Delta (\{e,h\})$.
Let $b' = \s_{de}(h)$; note that by the definition of $d$ this is well-defined,  $b' \in [1, n/d-1]$ is co-prime to $n/d$, and $b' \neq 1$ by the assumption $\s_{e}(h) \neq d$.
By Lemma \ref{aux_lem_higherorder}, we know that
\[
\min \Delta (\{e,h\})= \min \Delta (\{de,h\}) =
\min \Delta (\{e' ,b'e'\})
\]
with $e'=de$.
And, by  Remark \ref{rem_2el}, we also know that $\{e' ,b'e'\}$ is not half-factorial.
Thus,  $\min \Delta (\{e' , b'e'\}) \ge  \min \Delta(G_0) \ge (2n^2)^{1/3}$.
Again, by Theorem \ref{aux_thm_2elements}, applied to the group of order $n/d$ generated by $e'$, and an argument similar to the one before, we get
\[
b'=\frac{n/d-c_1^{(b')}}{c_2^{(b')}}
\]
with some positive integers $c_1^{(b')}$ and $ c_2^{(b')}$ that are smaller than  $(n/2)^{1/3}/d$.
We observe that
\[
\s_e(h)= d \s_{de}(h) =\frac{n-dc_1^{(b')}}{c_2^{(b')}}.
\]

Finally, we consider $\min \Delta (\{g,h\})$.
We set
\[
b''= \s_{dg}(h),
\]
and get as above
\[
\min \Delta (\{g,h\})= \min \Delta (\{dg,h\})=
\min \Delta (\{g', b''g' \})
\]
where $g'=dg$. And, again,
\[
b''=\frac{n/d-c_1^{(b'')}}{c_2^{(b'')}}
\]
with some positive integers $c_1^{(b'')}$ and $c_2^{(b'')}$ smaller than
$(n/2)^{1/3}/d$.

We compute $\s_g(h)$ in two ways. We have, as for $\s_e(h)$,
\[
\s_g(h)=\frac{n-dc_1^{(b'')}}{c_2^{(b'')}}.
\]
Yet, since by \eqref{eq_psigma} we have $\s_g(h) \equiv \s_g(e) \s_e(h) \pmod{n}$ and by the just obtained results, we also have
\[
\s_g(h) \equiv
\left( \frac{n-c_2^{(g)}}{c_1^{(g)}} \right)
\left(\frac{n-dc_1^{(b')}}{c_2^{(b')}}\right) \pmod{n}.
\]

Using the two conditions for $\s_g(h)$ it follows that
\[
-dc_1^{(b'')}c_1^{(g)} c_2^{(b')} \equiv
 c_2^{(g)} d c_1^{(b')}c_2^{(b'')} \pmod{n}
\]
and so
$dc_1^{(b'')}c_1^{(g)} c_2^{(b')} + d c_2^{(g)}c_1^{(b')}c_2^{(b'')}$,
which is a positive integer, has to be at least $n$.
Yet, this contradicts the conditions on the sizes of the involved
parameters obtained above.
\end{proof}

The next result complements the just obtained one.
We consider sets containing two generating elements, and elements of a special form.
Note that the result is trivial for $n \le 4$. 

\begin{lemma}
\label{specialdivisors}
Let $G$ be a finite cyclic group.
Let $e,g\in G$ be two distinct generating elements of $G$ such that
$\min \Delta(\{e,g\})> |G|^{1/2}$.
Let $D \subset G$ such that for each $h\in D$
we have that $\s_e(h)=\s_g(h)$ and this common value is a divisor of $|G|$.
Then, at least one of the following two assertions holds:
\begin{enumerate}
\item For some positive integers $c_1,c_2$ such that the triple
$(c_1,c_2, \gcd( \sigma_e(D)))$ belongs to $M (|G|)$,
\[
\min \Delta(\{e,g\}\cup D) = \gcd\left\{ \frac{|G|-c_1-c_2}{c_1c_2}, \,
\left( \frac{|G|-c_1 -c_2}{c_2 |G| }\right) \gcd(\sigma_e(D)) \right\}.
\]
\item $\min \Delta(\{e,g\}\cup D)\le |G|^{1/2}  + \log |G| - 1$.
\end{enumerate}

Moreover, if $\s_e(D)$ is totally ordered with respect to divisibility,
then we are in case (i).
\end{lemma}

\begin{proof} We set $n=|G|$ and assume $n \ge 5$ as otherwise the result is trivial. By Theorem \ref{aux_thm_2elements},
we know that $\s_e(g)=(n-c_1)/c_2$ for positive integers $c_1$ and
$c_2$ such that
\[
\min \Delta(\{e,g\}) = \frac{n-c_1-c_2}{c_1c_2}.
\]
Our goal is to apply Lemma \ref{aux_lem_geroldinger} to determine $\min \Delta(\{ e, g \} \cup D )$.

We consider a minimal zero-sum sequence over $\{e,g\}\cup D$ that actually contains some element of $D$; so let
\[
A\in \mathcal{A}(\{e,g\}\cup D)\setminus \mathcal{A}(\{e,g\}).
\]
We note that $e^{c_1}g^{c_2} \nmid A$ (as sequences) since $e^{c_1}g^{c_2}$ is a non-empty zero-sum sequence and thus, $A$ being a minimal zero-sum sequence, would have to equal $A$, which in turn would contradict our assumption.

So, we know that $\mathsf{v}_g(A) < c_2$ or $\mathsf{v}_e(A) < c_1$.
First, we assume that $v = \mathsf{v}_g(A) < c_2$.
Let
\[
A = g^vF  
\]
and define $f_e(A)= g^ve^{\sigma_e(F)}$.
Note that $\sigma_e(A) = \sigma_e(f_e(A))$ and thus $f_e(A)$ is a zero-sum sequence (over $\{e,g\}$).

We, first, argue that if $f_e(A)$ is not a minimal zero-sum sequence, then we are in case (ii).
We write $F= \prod_{i=1}^{\ell}(d_ie)$ and consider the zero-sum sequence $B= Ae^{n\ell}$. It is evident that $1+\ell$ is a length of $B$.
Noting that
\[
B= f_e(A)\prod_{i=1}^{\ell} \bigl(e^{n-d_i} (d_ie) \bigr),
\]
we see that $\ell' +\ell$ is a length of $B$ for each $\ell' \in \Lo(f_e(A))$.

We establish an upper bound for $\max \mathsf{L}(f_e(A))$. Namely
\[
\max \mathsf{L}(f_e(A)) \le v + \frac{\s_e(F)}{n} = v + \ko(F)
\le c_2 + \mathsf{k}(G)\le n^{1/2}  + \log n,
\]
for the last inequality we use inequality \eqref{star33}. 
Now, since $1+\ell$ and $\max \mathsf{L}(f_e(A)) + \ell$ are lengths of $B$, it follows that if $\max \mathsf{L}(f_e(A))>1$, then $\min \Delta(\Lo(B))\le n^{1/2}  + \log n - 1 $.
Thus, in case $f_e(A)$ is not a minimal zero-sum sequence we get the bound claimed in (ii).

So, we may assume that for each minimal zero-sum sequence $A\in \mathcal{A}(\{e,g\}\cup D)\setminus \mathcal{A}(\{e,g\})$ with $\mathsf{v}_g(A) < c_2$ we have that $f_e(A)$ is a minimal zero-sum sequence.

Yet, for each minimal zero-sum sequence $A\in \mathcal{A}(\{e,g\}\cup D)\setminus \mathcal{A}(\{e,g\})$ with $\mathsf{v}_e(A) < c_1$, the exact same
argument works, interchanging the roles of $e$ and $g$ as well as $c_1$
and $c_2$.
However, there is one crucial point to observe.
Namely, denoting, for $A=e^wg^vF'$ with $F'\in \mathcal{F}(D)$,
the sequence  $e^wg^{v+\sigma_e(F')}$ by $f_g(A)$ -- the sequence obtained by the analogue of the replacement defining $f_e(A)$  -- we clearly have $\s_g(A)= \s_g(f_g(A))$. However, for our further analysis we actually need to understand the relation between $\s_e(A)$ and $\s_e(f_g(A))$ as we wish to apply Lemma \ref{aux_lem_geroldinger}
and we must not mix in our considerations $\s_e$ and $\s_g$, but consistently use one of the two.
Thus we observe that
\[
\sigma_e(f_g(A))-\sigma_e(A)= \sigma_e(F')
\left(\frac{n-c_1}{ c_2} - 1 \right).
\]
Note that $\gcd \sigma_e(D) \mid \sigma_e(F')$.

Thus, it follows that if all minimal zero-sum sequences remain minimal zero-sum sequences under the replacements, then (the equality by Lemma \ref{aux_lem_geroldinger})
\[
\min \Delta(\{e,g\}\cup D) = \gcd \left\{ \frac{\sigma_e(A)}{n} - 1 \colon A \in \mathcal{A}(\{e,g\}
\cup D) \right\}
\]
is a multiple of
\[
\gcd \left\{ \gcd \left\{ \frac{\sigma_e(A)}{n} - 1 \colon A \in \mathcal{A}(\{e,g\}) \right\}, \,
\gcd(\sigma_e(D)) \left( \frac{n-c_1 -c_2}{c_2 n} \right) \right\}.
\]

We observe that for each element $h \in D$, we have the minimal zero-sum sequence  $hg^{n-\sigma_e(h)}$, and that
\[\s_e(hg^{n-\sigma_e(h)}) = n \frac{n-c_1}{c_2}-\sigma_e(h)\frac{n-c_1-c_2}{c_2}\]
that is
\[\frac{\s_e(hg^{n-\sigma_e(h)})}{n}-1 =  \frac{n-c_1-c_2}{c_2}- \sigma_e(h)\frac{n-c_1-c_2}{nc_2}.\]
Moreover, we recall that by Lemma \ref{aux_lem_geroldinger} and Theorem \ref{aux_thm_2elements}
\[
\gcd \left\{ \frac{\sigma_e(A)}{n} - 1 \colon A \in \mathcal{A}(\{e,g\}) \right\} = \frac{n-c_1-c_2}{c_1c_2}.
\]
In combination this yields that $\min \Delta(\{e,g\}\cup D)$ indeed also divides
\[
\gcd\left\{ \frac{n-c_1-c_2}{c_1c_2}, \,
\left( \frac{n-c_1 -c_2}{c_2 n }\right) \gcd(\sigma_e(D)) \right\},
\]
establishing the result.

It remains to show the additional statement.
To this end we need to analyze under which conditions a minimal zero-sum sequence remains minimal
under $f_e$ or $f_g$.

It suffices to do so in one case, the other one being analogous.
Let $A$ be a minimal zero-sum sequence, containing an element of $D$, and assume that $\mathsf{v}_g(A) < c_2$.
(We use the notation $v$, $F$, and $f_e(A)$ with the same meaning as above.)
\medskip

\noindent
\textbf{Claim:}
The sequence $f_e(A)$ is a minimal zero-sum sequence if and only if
$\sigma_e(F)\le n$.
\smallskip

To establish this claim we first observe that if $\sigma_e(F)> n$, it is immediate that $f_e(A)$ is not a minimal zero-sum sequence.
Now, suppose $f_e(A)$ is not a minimal zero-sum sequence, say $f_e(A)= A_1A_2$ with non-empty zero-sum sequences $A_1$ and $A_2$.
Let $v_1$ and $v_2$ denote the multiplicities of $g$ in $A_1$ and $A_2$, respectively. We have  $v_1 + v_2 = v$.
Moreover, it follows that $g^{v_i}e^{n - v_i \s_e(g)}\mid A_i$ for $i\in \{1,2\}$; for the sake of formal correctness we observe that $n - v_i \s_e(g)$ is non-negative as $v_i \le v < c_2$ and $\s_e(g)=(n-c_1)/c_2$.
Thus, the multiplicity of $e$ in $f_e(A)$ is at least
$2n - (v_1 + v_2)\s_e(g)> n$; again, since $v<c_2$  and $\s_e(g)=(n-c_1)/c_2$.
Since $\vo_{e}(g(A))= \s_e(F)$, this establishes the claim.

Now, it suffices to note that if $\s_e(D)$ is totally ordered with respect to divisibility, we have $\sigma_e(F)\le n$ as otherwise $F$ would have a non-empty zero-sum subsequence.
\end{proof}

Now, we are ready to give the proof of the main result of this section.

\begin{proof}[Proof of Theorem \ref{2maximal}] Let $n=|G|$. To avoid some notational inconveniences, we assume -- this is no restriction  -- that $0\notin G_0$.
By Theorem \ref{aux_thm_2elements} we know that
$\s_e(g)=(n-c_1)/c_2$ with $c_1,c_2 \in \N$ such that
\[
\min \Delta(\{e,g\})= \frac{n-c_1-c_2}{c_1c_2}.
\]

Let $h\in G_0 \setminus \{e,g\}$. By Proposition \ref{3elements_improved} and in view of the condition on  $\min \Delta(G_0)$, we get that $\s_e(h)$ or $\s_g(h)$ is a divisor of $n$.
We assume $\s_e(h)=d$ is a divisor of $n$. This assumption introduces
an asymmetry in $e$ and $g$; yet, it actually dissolves in the course
of the argument.

By Lemma \ref{smallelements} and again by the condition on $\min \Delta(G_0)$, we know that if $d <(2n^2)^{1/3}$, then necessarily $\min \Delta(G_0\setminus \{h\})= \min \Delta(G_0)$.
Repeatedly applying this argument, we may assume without restriction that $G_0$ does not contain an element of this form.

Now, suppose that $d\ge (2n^2)^{1/3}$.
We consider the set $\{h,g\} \subset G_0$.
Its minimal distance is a multiple of the minimal distance of $G_0$.
We will assert that this is only possible if the former is $0$, that is $\{h,g\}$
is a half-factorial set.

By Lemma \ref{aux_lem_higherorder}, we get that $\{h,g\}$
has the same minimal distance as $\{h,dg\}$. Yet this is a subset
of $\langle dg \rangle$, i.e., a finite cyclic group of order
$n/d$, and thus 
\[
\min \Delta (\{h,dg\})\le n/d - 2<  (2n)^{1/3},
\] 
where we used the first statement of \eqref{GeroHami} for the first inequality.

Thus, $\min \Delta (\{h,dg\})=0$, which is only possible if $h=dg$ (cf.~Remark \ref{rem_2el}).
In particular, this means that $de=dg$, which resolves the above mentioned asymmetry.

Thus we have (without restriction) $G_0 = \{e,g\}\cup D$ such that for each $h \in D$ we have $\s_e(h) \mid n $ and $\s_e(h)e = \s_e(h)g$; and so also $\s_e(h)=\s_g(h)$.
The minimal distance of sets of this form has been determined in Lemma \ref{specialdivisors}, and it is precisely of the claimed form.

To get the `moreover'-statement it suffices to inspect the proof and to note that if
$\s_e(h)e=\s_e(h)g$ for each $h \in D$, then $\gcd (\s_e(D))e = \gcd (\s_e(D)) g$.
\end{proof}

We now come to the proof of the converse statement.

\begin{proof}[Proof of Proposition \ref{prop_2maximal}]
The proof is essentially a direct consequence of Lemma \ref{specialdivisors};
it merely remains to check two details.

First, as pointed out after Definition \ref{def_sets} the integrality of the fraction
$\frac{n-c_1-c_2}{c_1c_2}$ guarantees that for $e$ a generating element of $G$ the element
\[
g = \left( \frac{n-c_1}{c_2} \right) e
\]
is also generating.

Second,  the integrality of $d (n-c_1 -c_2)/(c_2 n)$ is equivalent
to $d(g-e)=0$, and thus $de=ge$. Hence, we can apply
Lemma \ref{specialdivisors} to the set $\{e,g,de\}$,
and get the required minimal distance.
\end{proof}

We underline the fact that the proof of Theorem \ref{2maximal}
is non-constructive, in the sense that it does not allow to extract
the structure of the set $G_0$; this is due to the non-constructiveness
of Lemma \ref{smallelements}. The following lemma partly resolves
this issue. Again, the result is trivial for groups of order at most $4$.

\begin{lemma}
\label{smalldetailed}
Let $G$ be a finite cyclic group.
Let $e,g\in G$ be two distinct generating elements of $G$ such that
$\min \Delta(\{e,g\})> |G|^{1/2}$, that is by Theorem \ref{2maximal}
\[
\s_e(g)=\frac{|G|-c_1}{c_2}
\]
for some $(c_1,c_2,|G|)\in M(|G|)$.
And, let $h\in G$.
Suppose  $\mD{\{e,g,h\}} \ge (2|G|^2)^{1/3}$. 
Then, $\s_{e}(h)=d$ or $\s_{g}(h)=d$ for some divisor $d$ of $|G|$, and this divisors verifies $de = dg$ or $d\mid \gcd(c_1, c_2)$.
\end{lemma}

\begin{proof}
Let $n=|G| \ge 5$. If neither $\s_{e}(h)$ nor $\s_g(h)$ is a divisor of $n$ we are done by Proposition \ref{3elements_improved}. Since the condition $de = dg$ or $d\mid \gcd(c_1, c_2)$ is symmetric in $e$ and $g$, recall that $e= \frac{n-c_2}{c_1}g$, we may thus assume that $h=de$ for some divisor $d\mid n$.
Suppose $dg \neq de$.
By Lemma \ref{aux_lem_higherorder} the minimal distances of $\{de,g\}$ and $\{de,dg\}$ are equal,
and -- by Remark \ref{rem_2el} it is non-zero -- at least $(2n^2)^{1/3}$, the minimal distance of
$\{e,g,de\}$.

Note that $\{de,dg\}$ is contained in a cyclic group of order $n/d$.
Let $b=\s_{de}(dg)$. It is clear that $\gcd(b,n/d)=1$ as $g$ is a generating element of $G$.
By Theorem \ref{aux_thm_2elements}, it follows that
$b= (n/d - b_1)/b_2$ for some  $(b_1,b_2,n/d)\in M(n/d)$.

We observe that $b = (d \s_e(g) - kn)/d$ for a suitable integer $k$.
Thus, one condition on $c_1,c_2$ and $b_1,b_2$ we obtain is that
\begin{equation}
\label{eq_smalldetailed}
\frac{n-c_1}{c_2} - k \frac{n}{d} = \frac{n/d - b_1}{b_2}.
\end{equation}
Yet this is is not all. In addition, $b_1,b_2$ and $c_1,c_2$
need to be sufficiently small to guarantee that the respective
minimal distances of $\{e,g\}$ and $\{de,g\}$,
which we know in terms of the $c_i$'s and the $b_i$'s, resp.,
by Theorem \ref{aux_thm_2elements}, are at least $(2n^2)^{1/3}$.
Multiplying the just obtained equation by $c_2b_2$, we get that
\[c_1b_2 \equiv c_2b_1 \pmod{n/d}.\]

We now prove that $c_1b_2$ and  $c_2b_1$ are in fact equal.
Assume to the contrary that they are not equal. It follows that at least one of
$b_1,b_2,c_1,c_2$ is  at least $(n/d)^{1/2}$.
If it is $b_1$ or $b_2$, we get
\[
\mD{\{de,g\}}=\mD{\{de,dg\}} =  \frac{n/d -b_1-b_2}{b_1b_2}< (n/d)^{1/2} \le n^{1/2}< (2n^2)^{1/3},
\]
a contradiction. So, assume it is $c_1$ or $c_2$. This yields
\[
\mD{\{e,g\}} = \frac{n-c_1-c_2}{c_1c_2} <  \frac{n}{(n/d)^{1/2}}= (nd)^{1/2}.
\]
This does not right away give a contradiction,
as so far we have no information on the size of $d$ relative to $n$.
However, since $\mD{\{de,g\}}=\mD{\{de,dg\}}$ and $\{de,dg\}$ is contained in a cyclic group of order $n/d$ we have, by the general bound \eqref{GeroHami}, that $\mD{ \{de,g\}} \le n/d -2$. Thus, it follows
\[
\mD{\{e,g,de\}} < \min\{(nd)^{1/2}, n/d \} \le n^{2/3},
\]
contradicting our assumption.

We in fact thus have
\[c_1b_2 = c_2b_1.\]
Plugging this into \eqref{eq_smalldetailed},
we get after some computation
\[d - kc_2 = \frac{c_2}{b_2}.\]
Since as just established $c_2/b_2= c_1/b_1$, we also have \((d-kc_2) = c_1/b_1\),
implying that
\[
(d - kc_2) \mid \gcd(c_1,c_2).
\]
So, $d= d_0 + kc_2$ for some divisor $d_0$ of $\gcd(c_1,c_2)$.
We point out that $d_0 \mid d$, and that $b_i = c_i/d_0$ for $i\in \{1,2\}$.

Thus we can rewrite
\[
b = \frac{n/d-b_1}{b_2} = \frac{n/d - c_1/d_0}{c_2/d_0}
\]
and the minimal distance of $\{de,dg\}$ and thus  $\{de,g\}$ as
\[
D_2 = \frac{n/d - c_1/d_0 - c_2/d_0}{ c_1c_2/d_0^2}
=  \frac{d_0^2}{d}  \left( \frac{n - (c_1 + c_2) d/d_0}{ c_1c_2} \right).
\]
We put $D_1 = \frac{n - c_1 -c_2}{c_1c_2}$, the minimal distance
of $\{e,g\}$, and compute
\[
D = d_0 D_1 - \frac{d}{d_0} D_2.
\]
Recall that the minimal distance of $\{e,de,g\}$ divides $\gcd\{D_1,D_2\}$, and thus $D$ is a multiple of $\mD{\{e,de,g\}}$.

We note that
\[D= \frac{(d-d_0)(c_1+c_2)}{c_1c_2},\] and this is at most $2d$.
Since as mentioned above $n/d$ is at least $(2n^2)^{1/3}$, we get
that $2d$ is less than $(2n^2)^{1/3}$.
So, the only way that $D$ can be a multiple of $\mD{\{e,de,g\}}$ is
that  $D=0$. This means that  $d= d_0$. In view of
$d_0 \mid \gcd(c_1,c_2)$, as indicated above, the claim follows.
\end{proof}

\section{Main results}
\label{sec_proofs}

In this section, we formulate and prove our main results.
We start with a lemma, whose proof is mainly a summary of
the preceding results; for clarity and to simplify the subsequent discussion, we include a few redundant points.

\begin{lemma}
\label{main_lem_ul}
Let $G$ be a finite cyclic group.
\begin{enumerate}
\item For each divisor $m$ of $|G|$, we have
\[
\In{m}\subset \Dast{G},
\]
\item We have the inclusion
\[
\Dast{G} \cap \N_{> |G|^{1/2}} \subset
\bigcup_{m\mid  |G|} \Hud{m},
\]
\item We have the inclusion
\[
\Dast{G} \cap \N_{\ge (2|G|^2)^{1/3}} \subset
\In{|G|} \cup   \bigcup_{m \mid  |G|,\ m \neq |G|} \Hud{m} .
\]
\end{enumerate}
\end{lemma}

\begin{proof}
(i) It suffices to apply Proposition \ref{prop_2maximal} for each subgroup
of $G$.

(ii) Let $G_0 \subset G$ with $\mD{G_0} > |G|^{1/2}$.
By our condition on $\mD{G_0}$ we can apply Theorem \ref{prop_2subset} and thus get that there exists a subset $G_2\subset G_0$
of cardinality $2$ that is not half-factorial.
We are interested in the minimal distance of $G_2$.
By Lemma \ref{aux_lem_higherorder} we may assume that the two elements of $G_2$ have the same order; denote it by $m$. 
By Theorem \ref{aux_thm_2elements} it follows that $\mD{G_2}$ equals $(m-c_1-c_2)/(c_1 c_2)$
for some $(c_1,c_2,m)\in M(m)$.
Since $\mD{G_0}\mid \mD{G_2}$, we get that $\mD{G_0}\in \Hud{m}$, and the claim follows.

(iii) We start as in (ii). If $G_2$ does not consist of two elements
of order $|G|$, we know that $m$ as in (ii) is not $n$ and $\mD{G_0}\in \Hud{m}$ for a proper divisor $m$. If, however,
$G_2$ consists of two elements of order $|G|$, we can apply
Theorem \ref{2maximal} to get $\mD{G_0}\in \In{n}$.
\end{proof}

First, we give a description of the large elements of $\Dast{G}$ for $G$ a cyclic group of prime-power order; for earlier results see \cite[Corollary 1]{geroldinger90b} (essentially it asserts the inclusion in \ref{main_lem_ul}.2).

\begin{theorem}
\label{thm_pgroups}
Let $G$ be a cyclic group of prime-power order.
Then,
\[
\Dast{G}\cap  \N_{\ge (2|G|^2)^{1/3}}  =
\bigcup_{m\mid  |G|} \In{m} \cap \N_{\ge (2|G|^2)^{1/3}} .\]
\end{theorem}
\begin{proof}
Let $n$ denote the order of $G$.
The claim is trivial for $n \le 13$. 
One inclusion is merely Lemma \ref{main_lem_ul}.1; we establish the other one.
Let $G_0\subset G$ with $\min \Delta(G_0)\ge (2|G|^2)^{1/3}$, and
we have to show that it is an element of $\In{m}$ for some $m \mid |G|$.

Again, by Theorem \ref{prop_2subset} $G_0$ has a subset  $G_2=\{f_1,f_2\}$
of cardinality two that is not half-factorial.
Let $G_2$ be chosen such that $(\ord f_1, \ord f_2)$ is maximal
in the lexicographic order among all such two-element sets.
Note that this implies $\ord f_2 \le \ord f_1$.

First, we assert that $G_0 \subset \langle f_1 \rangle$.
Assume not. By the assumption on $G$, this means there is some $g\in G_0$ with $\ord g > \ord f_1$.

By our assumption on $(\ord f_1, \ord f_2)$ it follows that $\{g,f_1\}$ is half-factorial, since $(\ord g, \ord f_1)$ is greater than $(\ord f_1, \ord f_2)$.
By Lemma \ref{aux_lem_higherorder} and Remark \ref{rem_2el} it follows that, for
$d\mid n$ such that $\ord(dg)=\ord f_1 $, we have $dg = f_1$.
We assert that $\{g,f_2\}$ is not half-factorial, which violates again the maximality assumption on $(\ord f_1, \ord f_2)$.
To see this observe, for example, that by Lemma \ref{aux_lem_higherorder}
$\mD{\{g,f_2\}}=\mD{\{f_1,f_2\}}$.

So, we have established that $G_0 \subset \langle f_1 \rangle$.
Suppose $h\in G_0$ with $\ord h> \ord f_2$.
It follows that $\{f_1, h\}$ is half-factorial.
Thus, again, $d'f_1 =h$ for some $d'\mid n$.
It follows that the set of all elements of $G_0$ of order greater
than $f_2$, let us denote it by $G_3$, is contained in
$\{df_1\colon d\mid n\}$.
In particular, as $n$ is a prime-power, it follows,
by repeated applications of Lemma~\ref{aux_lem_higherorder},
that the minimal distance of $G_0$ is equal to the one of
$G_0' = (G_0 \setminus G_3) \cup \{d''f_1\}$ where $d''\mid n$
such that the order of $d'' f_1$ is equal to the one of $f_2$;
of course $d''$ might be equal to $1$ (and $G_3$ empty).

Yet, $G_0'$ fulfills the condition of Theorem \ref{2maximal} with respect to the group $\langle G_0' \rangle$.
\end{proof}

For finite cyclic groups of general order, the problem is more complex; the just given argument does not carry over directly.
We are only able to obtain a result for a smaller range of values. 

\begin{theorem}
\label{main_thm_direct}
Let $G$ be a finite cyclic group of order at least $2000$.
Then
\[
\Dast{G} \cap \N_{\ge \frac{|G|}{10}} =
  \bigcup_{m\mid |G|} \In{m} \cap \N_{\ge \frac{|G|}{10}} .
\]
\end{theorem}

It would not be too complicated to determine $\Dast{G} \cap \N_{\ge c |G|}$ 
with $c$ somewhat smaller than $1/10$.
It is not clear to us though until what point this type of result does hold. 
It seems possible that at some point  elements from $\Hud{m} \setminus \In{m}$ for $m \neq |G|$ will appear; 
this believe is based on the example given in Lemma \ref{lem_n-68}, 
yet as shown in Lemma \ref{lem_n-410} certainly not all those elements appear. 

For the largest elements of $\Dast{G}$, the condition
on $|G|$ could be considerably weakened (see Remark \ref{rem_sizen0}), 
and the problem of determining large elements of $\Dast{G}$ for a fixed, not too large, cyclic group $G$
is fairly manageable. It seems thus feasible, though possibly very tedious,
to remove this condition.

\begin{proof}
Let $n$ denote the order of the group. We observe that by our assumption on $n$, we have that $n/10 \ge (2n^2)^{1/3}$.
By Lemma \ref{main_lem_ul} we already have considerable knowledge on $\Dast{G}$, in particular regarding very large elements.
Indeed, by this result the only values that might be contained in $\Dast{G}$, yet of which we do not yet know so are elements of $J(m)\setminus I(m)$ for a proper divisor $m$ of $n$.  

We assert that the only  integers of size at least $n/10$ fulfilling this are 
\[
\frac{n-4}6,\quad \frac{n-4}8,\quad \frac{n-6}8\quad \text{ and }\quad \frac{n-6}9
\] 
(of course under the implicit assumption that these values actually are integers).

Since $J(m)\setminus I(m)$ contains only proper divisors of elements from  $I(m)$, it follows that every element in this set is of size at most $\max I(m)/2 < m/2$.
Thus, it suffices to consider $m \in \{n/2, n/3, n/4 \}$.

First, consider $m=n/2$. We consider the set $J(n/2) \cap \N_{\ge \frac{n}{10}}$. 
We have the elements $(n-4)/2$, $(n-4)/4$, $(n-4)/6$, $(n-4)/8$ from the choice $c_1= c_2 = 1$ and divisors. 
We have the elements $(n-6)/4$, $(n-6)/8$ corresponding to $c_1= 1$ and $c_2= 2$, and a divisor of it. 
We have the elements $(n-8)/6$, $(n-8)/8$, $(n-10)/8$,   corresponding to $c_1= 1$ and $c_2= 3$, $c_1= c_2 = 2$, $c_1= 1$ and $c_2= 4$, respectively.
All but $(n-4)/4$, $(n-4)/6$, $(n-4)/8$, $(n-6)/8$ are clearly in $I(n/2)$. Now, $(n-4)/4$ is also in $I(n/2)$ stemming from $(1,1,n/4) \in M(n/2)$. 

Now, consider $m=n/3$. We consider the set $J(n/3) \cap \N_{\ge \frac{n}{10}}$. 
We have the elements $(n-6)/3$, $(n-6)/6$, $(n-6)/9$, from the choice $c_1= c_2 = 1$ and divisors. 
We have the elements $(n-9)/6$ corresponding to $c_1= 1$ and $c_2= 2$. 
We have the elements $(n-12)/9$  corresponding to $c_1= 1$ and $c_2= 3$.
All but $(n-6)/6$, $(n-6)/9$ are clearly in $I(n/3)$. Now, $(n-6)/6$ is also in $I(n/3)$ stemming from $(1,1,n/6) \in M(n/3)$. 

Finally, consider $m=n/4$. We consider the set $J(n/4) \cap \N_{\ge \frac{n}{10}}$. 
We have the elements $(n-8)/4$ and $(n-8)/8$, from the choice $c_1= c_2 = 1$ and divisors.  We have the element $(n-12)/4$ corresponding to $c_1= 1$ and $c_2= 2$. 
Except for $(n-8)/8$, they are clearly in $I(n/3)$, and $(n-8)/8$ is also in $I(n/4)$ stemming from $(1,1,n/8) \in M(n/4)$.

This shows our assertion. We proceed to investigate the remaining elements. We see they are all in $I(n)$.   
We have that $(n-6)/8$ and $(n-6)/9$ are in $I(n)$ stemming from  $(2,4,n)$ and $(3,3,n)$, respectively, in $M(n)$.
We note that $(n-4)/8$ is an element of $I(n)$, too, as in this case $4 \mid n$ and $(2, 2, n/4) \in M(n)$.
Similarly $(n-4)/6$ is an element of $I(n)$, as in this case $2 \mid n$ and $(1, 3, n/2) \in M(n)$.
\end{proof}

The proof directly yields the more technical result (the condition $c_0\le 0.5$ is there to exclude very small $n$).
\begin{remark}
\label{rem_sizen0}
Let $G$ be a finite cyclic group of order $n$ and let $0.1 \le c_0 \le 0.5$.  
If $n  \ge 2 c_0^{-3}$, then  
\[
\Dast{G} \cap \N_{\ge c_0n} = \bigcup_{m\mid n} \In{m} \cap \N_{\ge c_0n}. 
\]
\end{remark}

With this remark at hand the proof of Theorem \ref{thm_directweak} is direct. 
\begin{proof}[Proof of Theorem \ref{thm_directweak}]
The condition $n \ge 250$ is the one implied by Remark \ref{rem_sizen0} for $c_0=1/5$. The set is just $\bigcup_{m\mid n} \In{m} \cap \N_{\ge n/5}$.
\end{proof}

In the following lemma we give an example of a set $G_0$ with $\mD{G_0} = (n-6)/8$ that  differs from the one mentioned in the proof of Theorem \ref{thm_directweak}.  

\begin{lemma}
\label{lem_n-68}
Let $G$ be a finite cyclic group of order $n$ and assume $(n-6)/8$ is a natural number. 
Let $f$ denote a generating element of $G$. The minimal distance of   
\[
\{f,2f, (n/2) f , -4f\}
\] 
is $(n-6)/8$. 
\end{lemma}
\begin{proof}
By Theorem \ref{aux_thm_2elements}, applied to the group generated by $2f$ we know that 
\[
\mD{ \{  2f,  -4f \}}= \frac{n/2 - 3}{2} = \frac{n - 6}{4}.
\]
Then, by Lemma \ref{aux_lem_higherorder} we know that 
$\mD{ \{ f, 2f , -4f \}} = \mD{ \{ 2f, -4f \}}$ and also $\mD{ \{  2f, (n/2) f , -4f \}} = \mD{ \{  2f,  -4f, 0 \}} = \mD{ \{ 2f, -4f \}}$; and, also the simple fact  $ \mD{ \{ f, (n/2) f  \}}=0 $.
We wish to apply Lemma \ref{aux_lem_geroldinger}. 
By the just made reasoning, it suffices to consider minimal zero-sum sequences containing both $f$ and $(n/2)f$ and at least one of the elements $2f, -4f$. 
Of course, $(n/2)f$ thus has to appear exactly once.

Let 
\[
A = f^u\ (2f)^v\ (-4f)^w\ ((n/2)f)
\]
be such a minimal zero-sum sequence.
If $w=0$ it is easy to see that $\s_f(A)=n$.
Suppose $w>0$. It follows that $u + 2v < 4$, since otherwise we would get a proper zero-sum subsequence. 
Moreover, by the congruence condition on $n$, it follows that in fact $u+2v$ is congruent to $1$ or $5$ modulo $8$, that is $u+2v = 1$. 
Since we now know that $u=1$ and $v=0$, we also get that $w= (n+2)/8$.
Therefore $\s_f(A)= n(n+2)/8$.
Consequently, by Lemma \ref{aux_lem_geroldinger},
\[
\begin{aligned}
& \mD{ \{ f, 2f, (n/2) f , -4f \}}   \\
 = &\gcd \left \{ \frac{\s_f(A)}{n} -1 \colon A \in \ac \bigl(\{ f, 2f, (n/2) f , -4f \} \bigr) \right\} \\ 
 = &\frac{n-6}{8}.
\end{aligned}
\]
\end{proof}

We  now turn to the inverse problem, that is the problem of characterization of sets with a large minimal distance.
We start with a result for groups of prime order where the problem is simpler than in the general case. 
We note that the result is void for groups of order less than $13$.

\begin{theorem}
\label{main_thm_prime}
Let $G$ be a finite cyclic group of order $n$, and assume that $n$ is prime.
Let $G_0 \subset G$ and  $d \in \Dast{G} \cap \N_{\ge (2n^2)^{1/3}}$.
Then $\mD{G_0}=d$ if and only if
\[
\{ f , ((n - c_1)/c_2) f \} \subset G_0 \subset \{ 0, f, ((n - c_1)/c_2) f \}.
\]
with $c_1,c_2 \in \N$ such that $d= (n- c_1 -c_2)/(c_1c_2)$ and some $f\in G \setminus \{ 0 \}$.
\end{theorem}
\begin{proof}
By Theorem \ref{prop_2subset} we know that $G_0$ contains at least two non-zero elements (that evidently have order $n$) and by Theorem \ref{2maximal} it cannot contain more than two non-zero elements and the relation among these two non-zero elements is as claimed. By Proposition \ref{prop_2maximal} this condition is also sufficient. The claim is established.
\end{proof}

The following result for finite cyclic groups in general gives a complete characterization of sets with minimal distance at least $|G|/5$. To avoid confusion, we point out that we chose this formulation of the result, rather than an `if-and-only-if' formulation, since in this way we can deal with congruence conditions on the order of the group implicitly. The values that appear in this result are the elements in $\Delta^{\ast}(G)$ that exceed the threshold $|G|/5$, as determined in Theorem \ref{thm_directweak}.

As can be seen by inspecting the proof, the difficulty of answering the inverse problem is not at all uniform for all the values. Indeed, for certain considerably smaller values we can also solve the inverse problem (see Lemmas \ref{lem_n-q-1q}, \ref{lem_n-2qq}, and \ref{lem_n-2qq2}). In this sense, the value $|G|/5$ is not the limit of our method.

\begin{theorem}
\label{main_thm_inverse}
Let $G$ be a finite cyclic group of order $n$ and let $G_0 \subset G$.
Suppose $n \ge 250$. For a suitable generating element $f$ of $G$,
the following assertions hold true:
\begin{enumerate}
\item If $\mD{G_0} =  n-2$, then
$\{f,-f\} \subset G_0 \subset \{f, -f, 0\}$,
\item If $\mD{G_0} = (n-2)/2$, then
$\{f, -f, (n/2) f\}\subset G_0 \subset \{f, -f, (n/2) f, 0\}$,
\item If $\mD{G_0} = (n-3)/2$, then
$\{f, -2f\}\subset G_0 \subset \{f, -2f, 0\}$,
\item If $\mD{G_0} = (n-4)/2$, then
\[
\{f, -2f\}\subset G_0\subset \{f, 2f, -2f, 0\}
\]
or
\[
\{2f, -2f\} \subset G_0\subset \{2f, -2f,  0\},
\]
or, in case $4 \nmid n$, 
\[
\{2f, -2f\} \subset G_0\subset \{2f, -2f, (n/2)f, 0\},
\]
\item If $\mD{G_0} = (n-4)/3$, then
$\{f, -3f\}\subset   G_0 \subset \{f, -3f, 0\}$,
\item If $\mD{G_0} = (n-6)/3$, then
\[
\{f, -3f\} \subset G_0 \subset \{f, 3f , -3f, 0 \}
\]
or
\[
\{3f, -3f\} \subset G_0 \subset \{3f, -3f, 0\},
\]
or, in case $9 \nmid n$, 
\[
\{3f, -3f\} \subset G_0 \subset \{3f, -3f, (n/3) f, 0\},
\]
\item If $\mD{G_0} = (n-4)/4$, then 
\[
\{f,  ((n-2)/2)f\}\subset G_0\subset
\{f,  ((n-2)/2)f,  2f, -2 f, (n/2) f, 0\}
\]
or
\[
\{f,-2f,  (n/2)f\}\subset G_0\subset
\{f, 2f, -2 f, (n/2) f, 0\}
\]
or, in case $8\nmid n $,
\[
\{2f,-2f, (n/2) f\} \subset G_0 \subset \{2f, -2f, (n/4) f,(n/2) f,0 \},
\]
\item If $\mD{G_0} = (n-5)/4$, then
$\{f, -4f\}\subset   G_0 \subset \{f, -4f, 0\}$,
\item If $\mD{G_0} = (n-6)/4$, then
\[
\{2f,-4f\} \subset G_0 \subset \{2f, -4f,(n/2)f, 0\}
\]
or
\[
\{f,-4f\} \subset G_0 \subset \{f,2f, -4f, 0\}
\]
or
\[
\{f,((n-2)/2) f\} \subset G_0 \subset \{f , 2f, ((n-2)/2) f , (n/2) f, 0 \},
\]
\item If $\mD{G_0} = (n-8)/4$, then
\[
\{jf, -4f\} \subset G_0 \subset \{f, 2f, 4f, -4f, 0\},
\]
for $j\in \{1,2,4\}$ or, in case $8 \nmid n$,
\[
\{4f, -4f\} \subset G_0 \subset \{4f, -4f, (n/4) f, (n/2) f, 0\}.
\]
\end{enumerate}
Moreover, in all cases, the sets actually have these minimal distances,
and there are no other sets with this minimal distance (except for the choice of the generating element).
\end{theorem}

The reason why in this result we only need to assume that the order is at least $250$ as opposed to $2000$ in Theorem \ref{main_thm_direct} is explained in Remark \ref{rem_sizen0}; indeed, using this remark we could specify bounds for each case. 

We split off parts of the proof in technical lemmas.
We point out that the two lemmas below are void for groups of order at most $7$, since we need at least $1 \le |G|^{1/3}/2$.

\begin{lemma}
\label{lem_n-q-1q}
Let $G$ be a finite cyclic group of order $n$ and $G_0 \subset G$.
Let $q$ be a divisor of $n-1$ with  $q \le n^{1/3}/2$. 
The following statements are equivalent:
\begin{enumerate}
\item $\mD{G_0}=(n-q-1)/q$,
\item $\{f,-qf\} \subset G_0 \subset \{0,  f, -qf\}$ for some generating element $f$.
\end{enumerate}
\end{lemma}

\begin{proof}
Suppose $n \ge 7$ as the result is void otherwise.
To see that (i) implies (ii), suppose that $G_0$ has
the required minimal distance.
By our assumption on the size of $q$, and thus $\mD{G_0}$, we can apply Theorem \ref{prop_2subset}. That is,  there exists a subset
$G_2=\{f_1,f_2\} \subset G_0$ such that $\Delta(G_2)\neq \emptyset$.
Clearly, 
\[
\mD{G_2} \ge \mD{G_0} \ge \frac{n-q-1}{q}.
\] 

It follows by Lemma \ref{aux_lem_higherorder} and
Theorem \ref{aux_thm_2elements} that  
\[\frac{n-q-1}{q} \Big\vert \frac{m - c_1 - c_2}{c_1 c_2}\]
for some divisor $m\mid n$ and $c_1, c_2 \in [1,m]$. 
Setting $m=n/d$, this yields 
\[\ell c_1c_2 \frac{(n-q-1)}{q} =  (n/d -c_1-c_2)\]
for some integer $\ell$, and further 
\begin{equation}
\label{equival33} 
\frac{n-d(c_1+c_2)}{dc_1 c_2} = \frac{l (n-q-1)}{q}.
\end{equation}

We assert that
\begin{equation}
\label{approuver} 
\ell  c_1c_2 (q+1) = q  (c_1+c_2).
\end{equation}
Indeed \eqref{equival33} implies
\[d c_1c_2 l (q+1) \equiv q d (c_1+c_2) \pmod{n}.\]
Let us show that both sides of this inequality are in the range $[1,n]$, and thus equal (as integers), which after division by $d$ will 
imply \eqref{approuver}. On the one hand, we have
\[\frac{n - dc_1 - dc_2}{dc_1 c_2} \ge  \mD{G_0} \ge \frac{n-q-1}{q} \ge 2 n^{2/3} -1 - \frac{1}{q} \ge 2 n^{2/3} -2 \ge n^{2/3},\]
and thus
\begin{equation}
\label{dc1c2}
dc_1c_2  \le \frac{n - dc_1 - dc_2}{n^{2/3}} < n^{1/3}.
\end{equation}
Moreover, \eqref{equival33} again implies 
\[
\ell=  \frac{(n/d -c_1-c_2)q}{c_1 c_2 (n-q-1)}  \le \frac{q(n-2)}{n-q-1} \le 2q-1 \le n^{1/3}-1
\]
(here we use $2q+1 \le n$ and $n \ge 7$) and 
\[
q+1 \le \frac{n^{1/3}}2+1 \; , 
\]
thus
\[
d c_1c_2 l (q+1) \leq n^{1/3} (n^{1/3}-1) \left( \frac{n^{1/3}}2+1 \right) <n
\]
for $n \ge 7$. On the other hand, by \eqref{dc1c2}, $d(c_1+c_2) \le 2dc_1 c_2 < 2 n^{1/3}$; 
and $q \le n^{1/3}/2$, thus finally
\[
qd(c_1+c_2) <n^{2/3} <n.
\]
Equality \eqref{approuver} is proved.

From \eqref{approuver} it follows that $c_1c_2 < (c_1 + c_2)$ which is only possible if  $c_1=1$ or $c_2=1$, say $c_1 = 1$. 
It remains to determine $c_2$. We know $(q+1) \ell  c_2  = q (1+c_2)$, hence $\ell c_2 < (1+c_2)$, therefore $\ell = 1$, implying $c_2= q$.

Thus we have $m = n$ and $\{c_1, c_2\} = \{1,q\}$.
By Lemma \ref{smalldetailed},  we have, for any element $h \in G_0 \setminus G_2$, that  $\s_{f_1}(h)= \s_{f_2}(h) = d'$ with $d'f_1 = d'f_2$; 
note that $\gcd(c_1,c_2)= 1$ shows that this case of the lemma occurs. Now, Lemma \ref{specialdivisors} is applicable and implies $d'=n$, and the claim is established.

The converse claim, that (ii) implies (i), is clear
by Theorem \ref{aux_thm_2elements}.
\end{proof}

\begin{lemma}
\label{lem_n-2qq}
Let $G$ be a finite cyclic group of order $n$.
Let $q=1$ or a prime divisor of $n$, less than $n^{1/3}/2$.
Let $G_0 \subset G$.
The following statements are equivalent:
\begin{enumerate}
\item $\mD{G_0}=(n-2q)/q$.
\item $\{qf,-qf\} \subset G_0 \subset \{0,  qf, -qf, d f\}$
where
\[
d = 
\begin{cases}
n/q & \text{if }q^{2}\nmid n \\
n   & \text{otherwise}
\end{cases}.
\]
or $\{f,-qf\} \subset G_0 \subset \{0, f, qf, -qf\}$ for some generating element $f$.
\end{enumerate}
\end{lemma}
In case $d= n$ we have of course $df= 0$, and in the first description the set on the right-hand side is just of cardinality $3$. 

\begin{proof}
Suppose $n \ge 7$ as the result is void otherwise.
To see that (i) implies (ii), suppose that $G_0$ has
the required minimal distance. By our assumption on the size of $q$, we can apply Theorem \ref{prop_2subset} to get a subset $G_2=\{f_1,f_2\} \subset G_0$ that is not half-factorial.
We thus have $\mD{G_2}\ge \mD{G_0} = (n-2q)/q$.

We assume without restriction that $(\ord f_1, \ord f_2)$ is maximal (in the lexicographic order) among all non-half-factorial subset of $G_0$ with two elements; in particular $\ord f_1 \ge \ord f_2$.
We assert that $(\ord f_1, \ord f_2)$ is equal to $(n, n/q)$
or $(n/q,n/q)$.

By Lemma \ref{aux_lem_higherorder}
and
Theorem \ref{aux_thm_2elements}
it follows that
\[
\frac{n-2q}{q}  \Big \vert \;
\frac{m -c_1-c_2}{c_1c_2}
\]
for a divisor $m$ of $n$, and more precisely $m=\gcd\{\ord f_1, \ord f_2\}$, and $(c_1,c_2,m)\in M(m)$.

We assert that this is only possible if $m = n/q$; in particular, also  $\ord f_2 = n$ is impossible (except for $q=1$). 

Let $\ell \in \mathbb{N}$ such that
\begin{equation}
\label{eq_n/q}
\ell\left( \frac{n-2q}{q} \right) =
\frac{m -c_1-c_2}{c_1c_2}
\end{equation}
and let $d$ be the co-divisor of $m$ in $n$,  that is $m=n/d$.
It follows from \eqref{eq_n/q} that
\[
\ell d c_1c_2n - 2 \ell d c_1c_2 q = qn - qd(c_1+c_2),
\]
and thus $2 \ell d c_1c_2 q$ equals $qd(c_1+c_2)$ modulo $n$.
Yet, since $q < n^{1/3}/2$, we have that, say,
\[\frac{n -dc_1-dc_2}{dc_1c_2} \ge \frac{n-2q}{q} > n^{2/3},\]
so that $dc_1c_2 < n^{1/3}$; also note that $\ell < n^{1/3}$.
Thus, indeed, $2 \ell d c_1c_2 q = qd(c_1+c_2)$, and hence
$\ell d c_1c_2n = qn$.
In case, $q=1$ it follows immediately that $d=1$ and we are done. Assume $q>1$.
We have that exactly one of $c_1$, $c_2$, $\ell$, and $d$ equals $q$ while the other quantities are equal to $1$.
We observe that \eqref{eq_n/q} can only hold if $q=d$, and are done again.

First, assume 
\[
(\ord f_1, \ord f_2) = (n/q,n/q).
\]
It follows by Lemmas \ref{specialdivisors} and \ref{smalldetailed} that $f_1 = -f_2$ and that
$\langle f_1 \rangle \cap G_0 \subset \{f_1, f_2, 0\}$.
Suppose $h \in G_0 \setminus \langle f_1 \rangle$.
It follows by Lemma~\ref{aux_lem_higherorder} that
\[
\mD{h,f_1,f_2, 0}=\mD{qh,f_1,f_2, 0};
\]
note $qh \in \langle f_1 \rangle $.
By Lemmas \ref{specialdivisors} and \ref{smalldetailed}, we get that $qh \in \{f_1, f_2, 0\}$.
We observe that $G_0 $ cannot contain two distinct elements
of order $q$: this obvious for $q \le 2$ and for $q> 2$ we would otherwise, by \eqref{GeroHami}, have
$\mD{G_0}\le q-2$.
Moreover, note that $G \setminus \langle f_1 \rangle$
contains an element of order $q$ if and only if  $q^2 \nmid n$.
Finally, we observe that $h$ cannot have order $n$,
since otherwise $\{f_1,h\}$ or $\{f_2,h\}$ is not half-factorial,
which violates our assumption on the extremal choice of $f_1, f_2$.
This establishes our claim.
\smallskip

Second, assume 
\[
(\ord f_1, \ord f_2) = (n, n/q).
\]
It follows by Lemma \ref{aux_lem_higherorder} that
\[
\mD{f_1,f_2}=\mD{qf_1,f_2}
\]
and furthermore by Theorem \ref{2maximal} that
$f_2 = -qf_1$. In addition,  we get that
$\langle f_2 \rangle \cap G_0 \subset \{qf_1, f_2, 0\}$; to see this, one applies again Lemma \ref{aux_lem_higherorder} 
and then argues as above. 

Suppose that $h \in G_0 \setminus \langle f_2 \rangle$.
It follows that $\{f_1,h\}$ is half-factorial, since $\ord h \neq n/q$ (cf.~above).
And, since thus $\ord h \neq n$, we also get that $\{f_2,h\}$
is half-factorial.
In particular, $\s_{f_1}(h) \mid n$ and $\s_{f_2}(qh)\mid n/q$  (see Remark \ref{rem_2el}).
Thus, $h=df_1$ and
\[
qh= d'f_2= (-d'q)f_1
\]
for $d\mid n$ and $d'\mid n/q$.
Moreover, since $h \notin \langle f_2 \rangle$ it follows that
$q \nmid d$ and thus $d \mid n/q$.
Considering the two different representations for $qh$, we get that
$n \mid qd + qd'$ and thus
\[
n/q \mid (d+d').
\]
So, $d=d' = n/(2q)$ or $d=d'=n/q$.

Suppose first $d=n/(2q)$. We then consider
\[
A=h^q\ f_2^{n/(2q)},
\]
which is a minimal zero-sum sequence.
We note that
\[
\frac{\s_{f_1}(A)}{n} -1=\frac{n-2q}{2q},
\]
yielding a contradiction as it is not a multiple of the minimal distance of $G_0$.

Suppose now that $d=n/q$. We consider
\[
A'=f_1^x\ h\ f_2^{(x + n/q)/q}
\]
where $x\in [0,q-1]$ and $x \equiv -n/q \pmod{q} $.
We note that
\[
\frac{\s_{f_1}(A')}{n} -1=\frac{n + xq - q^2}{q^2},
\]
yielding a contradiction as it is not a multiple of the minimal distance of $G_0$.
In any case, the structure of the set is as claimed.

It remains to show the converse, that is (ii) implies (i).
Essentially, all arguments were already given in this proof; we add some brief explanation.
The sets $\{e, -qe\}$ and $\{qe,-qe\}$ both have the required minimal distance (cf.~above). By Lemma \ref{aux_lem_higherorder}, the minimal distance of $\{0,  qe, -qe, d e\}$, with $d$ as in the result, is equal to the one of $\{0,  qe, -qe, qd e\}=  \{0,  qe, -qe\}$, which in turn is equal to the one of $\{ qe, -qe\}$. Similarly, the minimal distance of $\{e, qe, -qe,0\}$ is equal to the minimal distance of $\{qe, -qe\}$.
Of course, this also determines the minimal distance of all `intermediate' sets, and the claim is established.
\end{proof}

It might be interesting to note that for other values of $x$ in
the above proof we do not get another `large' minimal distance,
but a `small' one since the greatest common divisor of
$(n+ xq -q^2)/q^2$  and $(n-2q)/q$ would be `small'.

We continue this type of investigations with another lemma.
We note that the lemma is void if the order of the group is less than $6^6$.
The condition that $q$ is odd below is necessary as can be seen by comparing the result to the one for $(n-4)/4$ in Theorem \ref{main_thm_inverse}.

\begin{lemma}
\label{lem_n-2qq2}
Let $G$ be a finite cyclic group of order $n$.
Let $q\mid n$ be a prime greater than $2$ and at most $n^{1/6}/2$,
and let $G_0 \subset G$.
The following statements are equivalent:
\begin{enumerate}
\item $\mD{G_0}=(n-2q)/q^2$,
\item we have:
\begin{itemize}
\item $ G_0 \subset \{0, e, qe, -qe, (n/q) e, ((n-q)/q) e\}$,
\item $G_0$ contains $\{e,-qe, (n/q) e\}$ or $\{e, ((n-q)/q) e\}$, and
\item $n/q \equiv 2 \pmod{q}$.
\end{itemize}
\end{enumerate}
\end{lemma}

\begin{proof}
Suppose $n \ge 6^6$ as the result is void otherwise.
To prove that (i) implies (ii), assume that $G_0$ has the claimed minimal distance.
We assert that $G_0$ has a subset $G_2$ of cardinality $2$ with minimal distance $(n-2q)/q^2$ or  $(n-2q)/q$, and that these two values are the only two values that can occur as minimal distance of a non-half-factorial subset of $G_0$ of cardinality $2$.

By Theorem \ref{prop_2subset} we know that $G_0$ has a non-half-factorial subset $G_2= \{f_1,f_2\}$ of cardinality two. And, its minimal distance is a multiple of $(n-2q)/q^2$.

By Lemma \ref{aux_lem_higherorder} and Theorem \ref{aux_thm_2elements} it follows that
\begin{equation*}
\mD{G_2} =  \frac{m - c_1 - c_2}{c_1 c_2}
\end{equation*}
for a divisor $m$ of $n$, and more precisely  $m = \gcd\{\ord f_1, \ord f_2\}$, and $(c_1,c_2,m)\in M(m)$.

We write $m= n/d$ and let $\ell$ be the integer such that
\begin{equation}
\label{eq_n-2qq2_1}
\ell \frac{n-2q}{q^2} = \frac{n/d - c_1 - c_2}{c_1 c_2}.
\end{equation}
It follows that
\begin{equation}
\label{eq_n-2qq2_2}
\ell d c_1c_2 n -  2  \ell d  c_1 c_2 q = q^2 n  - q^2 d (c_1 + c_2)
\end{equation}
and so
\[ 2  \ell d  c_1 c_2 q \equiv   q^2 d (c_1 + c_2) \pmod{n}.\]
By our assumption on the size of $q$ and the conditions on the sizes of $c_1,c_2, d$  and $\ell$ that follow from this assumption, we have in fact that 
\[ 2  \ell d  c_1 c_2 q = q^2 d (c_1 + c_2).\]
And then from this and \eqref{eq_n-2qq2_1}, we obtain that
\[\ell d c_1c_2 n = q^2 n \]
and so
\[\ell d c_1c_2  = q^2 . \]

Therefore, either exactly one of $\ell, d, c_1$, and $c_2$ equals  $q^2$ and the other three $1$, or exactly two of the four equal $q$ and the other two $1$.

Checking all possibilities we get that $c_1=c_2= q$ and $\ell = d = 1$, yielding $\mD{G_2}= (n-2q)/q^2$, and $c_1=c_2 = 1$ and   $\ell = d = q$, yielding $\mD{G_2}= (n-2q)/q$, are the only choices for which \eqref{eq_n-2qq2_1} holds, implying the claim.

\bigskip

We distinguish cases according to $(\ord f_1, \ord f_2)$, and we assume that $G_2$ is chosen in such a way that this pair is maximal in the lexicographic order among all subsets of $G_0$ of cardinality $2$ that are not half-factorial.

By the above argument we get that this pair is equal to $(n/q,n/q)$, $(n,n)$, or $(n,n/q)$.

First, assume 
\[
(\ord f_1, \ord f_2) = (n/q,n/q),
\]
that is $d=q$ and $c_1=c_2=1$, yielding a minimal distance of $(n-2q)/q$.
The argument is similar to the one in Lemma \ref{lem_n-2qq}.
It follows that $f_1 = -f_2$ and by Lemmas \ref{specialdivisors} and \ref{smalldetailed} that
$\langle f_1 \rangle \cap G_0 \subset \{f_1, f_2, 0\}$.
Suppose $h \in G_0 \setminus \langle f_1 \rangle$.
It follows by Lemma~\ref{aux_lem_higherorder} that
\[
\mD{h,f_1,f_2, 0}=\mD{qh,f_1,f_2, 0};
\]
note $qh \in \langle f_1 \rangle $.
Again, by Lemmas  \ref{specialdivisors} and \ref{smalldetailed}, we get that $qh \in \{f_1, f_2, 0\}$.
We observe that $G_0 $ cannot contain two distinct elements
of order $q$, since otherwise, by \eqref{GeroHami}, 
 $\mD{G_0}\le q-2$, a contradiction.
Moreover, note that $G \setminus \langle f_1 \rangle$
contains an element of order $q$ since $q^2 \nmid n$.
Finally, we observe that $h$ cannot have order $n$,
since otherwise $\{f_1,h\}$ or $\{f_2,h\}$ is not half-factorial,
which violates our assumption on the extremal choice of $f_1, f_2$.
Thus, we proved that in this case $\{-qf, qf\} \subset G_0 \subset \{-qf, qf, 0 , (n/q) f\}$ for some generating element $f$. 
Note that this set is not of the form given in the result, yet, also note that we did not actually assert that such a set $G_0$ has the required minimal distance; indeed, we show later it does not have it.

Second, assume 
\[
(\ord f_1, \ord f_2) = (n,n/q).
\]
Thus, $d=q$ and so $f_2= -qf_1$. As in the proof of Lemma \ref{lem_n-2qq} it follows that
$G_0\subset \{f_1,qf_1 , - qf_1, (n/q)f_1, 0\}$; the only modification necessary is to observe that $(n-2q)/(2q)$ is not a multiple of the minimal distance as $q \ne 2$.

Third, assume 
\[
(\ord f_1, \ord f_2) = (n,n).
\]
We have $c_1=c_2 =q$, that is $f_2= ((n-q)/q) f_1$.
We know by Proposition \ref{3elements_improved} that $G_0$ cannot contain a third element of order $n$.
Let $h \in G_0$ an element of order less than $n$. First, assume its order is not $n/q$.
Let us denote its order by $n/d'$.  Since $\{f_1,h\}$ and $\{f_2,h\}$ are both half-factorial,
it follows that $d'f_1 = h = d' f_2$. Thus, $d'$ and $d'(n-q)/q$ are congruent modulo $n$.
After some calculation, this implies that $n\mid 2qd'$, that is $d'$ is a multiple of $n/(2q)$, and thus equal to $n$, $n/2$, $n/q$, or $n/(2q)$.

We show that $d'$ cannot be equal to $n/2$ or $n/(2q)$. 
To this end it suffices to consider the following two minimal zero-sum sequences: 
\[
A_1=(\frac{n}{2}f_1)\ (\frac{n-q}{q}f_1)^{n/2}
\]
and
\[
A_2=(\frac{n}{2q}f_1)^{2q-1}\ (\frac{n-q}{q}f_1)^{n/(2q)}.
\]
Since $\s_{f_1}(A_1)/n -1= (n-2q)/(2q)$ and $\s_{f_1}(A_2)/n -1= (n-2q)/(2q^2)$ are not multiples of $(n-2q)/q^2$, we see that these values of $d$ are impossible.
Thus, $h= (n/q)f_1$ or $h=nf_1=0$.

It remains to consider elements of order $n/q$. It follows, considering the set consisting of $qf_2= -qf_1$ and this element, from the argument at the beginning of the lemma that an element order $n/q$ other than $qf_2$ in $G_0$ needs to be $qf_1$.

Thus, we get that 
\[
G_0 \subset \{ f_1, qf_1, (n/q)f_1, -qf_1, ((n-q)/q)f_1, 0\}.
\]

So, we see that in each case we get a subset of $\{ f_1, qf_1, -qf_1, (n/q)f_1, ((n-q)/q)f_1, 0\}$.
To complete the proof it thus suffices to investigate which subsets of this set have the required minimal distance.
Recall that since we need that $(n-2q)/q^2$ is an integer, we have that $n/q$ is congruent $2$ modulo $q$.
As always, the presence of $0$ in a set does not have an effect on the minimal distance so we actually can reduce to studying subsets of $\{ f_1, qf_1, -qf_1, (n/q)f_1, ((n-q)/q)f_1\}$. 

Let 
\[
G_0' \subset \{ f_1, qf_1, -qf_1, (n/q)f_1, ((n-q)/q)f_1\}
\] 
be a subset with $\mD{G_0'}= (n-2q)/q^2$. 

Since 
\[
\mD{\{qf_1, -qf_1, (n/q)f_1 \}} = \mD{\{qf_1, -qf_1 \}} =  \frac{n}{q} -2
\] 
we see that $G_0'$ contains an element of order $n$. We thus may assume it contains $f_1$; note that the set $\{ f_1, qf_1, -qf_1, (n/q)f_1, ((n-q)/q)f_1\}$  is symmetric under change of the `basis element' $f_1$ to  $((n-q)/q)f_1$.

The set $\{ f_1, qf_1, (n/q)f_1\}$ is half-factorial (see Proposition \ref{3elementshf}). Thus, we get that $G_0'$ contains $-qf_1$ or $((n-q)/q)f_1$.

Assume first it does not contain $((n-q)/q)f_1$, and thus contains $-qf_1$.
By Lemma \ref{lem_n-2qq}, we get that $\{ f_1, qf_1, -qf_1\}$ has minimal distance $(n-2q)/q$, thus $G_0'$ contains $(n/d)f_1$. We continue by asserting that $\{ f_1,  -qf_1, (n/q)f_1\}$ and also $\{ f_1, qf_1, -qf_1, (n/q)f_1\}$ have the required minimal distance $(n-2q)/q^2$, i.e. $G_0'$ could be either of these two sets. 

We proceed as in the proof of Lemma \ref{lem_n-2qq}. Let $A$ be a minimal zero-sum sequences over the set in question.
We investigate $\s_{f_1}(A)/n - 1$ and in view of the arguments already given we can reduce to the case that $A$ contains $f_1,  -qf_1, (n/q)f_1$ (and thus does not contain $qf_1$). 

Let us write 
\[
A= f_1^u((n/q)f_1)^v(-qf_1)^w. 
\] 
Since $A$ is a minimal zero-sum sequence it follows that $u + v n/q < n$; moreover it is a multiple of $q$. Also note that $u,v \le q-1$. We have $w = (u + v n/q)/q$ and, 
since $n/q \equiv 2 \pmod{q}$, that $u +2v \equiv 0 \pmod{q}$. It follows that in fact $u = q-2v$, that is $v < q/2$, since otherwise 
\[
f_1\ ((n/q)f_1)^{(q-1)/2}\ (-qf_1)^{(1 + n(q-1)/(2q))/q}
\] 
would be a zero-sum subsequence of $A$. From this we get that 
\[
\frac{ \s_{f_1}(A)}{n} - 1  =  \frac{q-2v + v(n/q) + (n-q)(q-2v + v n/q)/q}{n}-1 
				 =  \frac{(n-2q)v}{q^ 2}.
\]    
This shows the claim. 

Assume now the subset does contain $((n-q)/q)f_1$. The set $\{f_1,((n-q)/q)f_1\}$ has the required minimal distance
it thus suffices to show that 
\[
\{ f_1, qf_1, -qf_1, (n/q)f_1, ((n-q)/q)f_1\}
\] 
does not have a smaller minimal distance.

To this end we consider $\s_{f_1}(A)/n -1$ for all minimal zero-sum sequences, and show this is always a multiple of  $(n-2q)/q^2$.

We first mention for completeness two cases that we more or less considered already.

If $A$ only contains $f_1$, $qf_1$, $-qf_1$, and $(n/q)f_1$, then the claim follows by the argument given just above.

If $A$ only contains $((n-q)/q)f_1$, $qf_1$, $-qf_1$, and $(n/q)f_1$, then the claim follows first considering $\s_{((n-q)/q)f_1}$ 
in combination with the above claim, and expressing $\s_{((n-q)/q)f_1}$ in terms of $\s_{f_1}$ using \eqref{eq_psigma} (note that  $f_1 = ((n-q)/q)  (((n-q)/q)f_1)$ ).

So, we can reduce to considering $A$ that actually contain both $f_1$ and $((n-q)/q)f_1$. It follows that $A$ cannot contain both $qf_1$ and $-qf_1$. Without restriction we can assume it does not contain $-qf_1$; otherwise, we could exchange the role of $f_1$ and $(n-q)/q)f_1$, and use again \eqref{eq_psigma} to complete the argument. 

If $A$ contains only $f_1$, $((n-q)/q)f_1$,  and $(n/q)f_1$ the claim follows from Lemma \ref{specialdivisors}.
So, we can assume $A$ contains $qf_1$, too.

This implies that the multiplicity of $((n-q)/q)f_1$ is less than $q$; otherwise $(qf_1)(((n-q)/q)f_1)^q$ would be a zero-sum subsequence of $A$.

If $A$ does not contain $(n/q)f_1$, this implies $\s_{f_1}(A)= n$, since there is no zero-sum-free sequence over $f_1, qf_1$ with $\s_{f_1}$-value greater than $n$ while the contribution of the elements $((n-q)/q)f_1$ to $\s_{f_1}$ is at most $((n-q)/q)(q-1) < n$.

So, we can assume that $A$ contains all of $f_1$, $qf_1$, $((n-q)/q)f_1$,  and $(n/q)f_1$.
We denote the multiplicities of the elements by $v$, $v_q$, $w$ and $u$ respectively.
As mentioned above we have $w < q$. Of course, $u < \ord((n/q)f_1) = q $, too. Moreover, we can assume that $v < q$; otherwise we could pass from $A$ to the minimal zero-sum sequences $(qf_1)f_1^{-q}A$, which does not change $\s_{f_1}$. Finally, $v_q < \ord(q f_1) = n/q$.

This directly implies that 
\[
\s_{f_1} (A) \le (q-1) + (n/q -1)q + (q-1)n/q + (q-1)(n-q)/q < 3n.
\]
Since  $\s_{f_1} (A)$ is a multiple of $n$, and if it is equal to $n$ we are done, it remains to exclude the case that it is  $2n$. So, assume $\s_{f_1} (A)= 2n$. We show that $A$ cannot be a minimal zero-sum sequence.

To this end it suffices to show that there exist 
\[
v_1 + v_2 = v,\quad  u_1 + u_2 = u,\quad w_1 + w_2 = w,
\] 
all non-negative integers, such that
\[
v_i + u_i \frac{n}{q}  + w_i \frac{n-q}{q} \le n\quad \text{ and }\quad v_i + u_i  \frac{n}{q}+ w_i \frac{n-q}{q} \equiv 0 \pmod{q}\ \text{ for } i=1,2.
\]
Recall that $n/q$ is equal to $2$ modulo $q$. So, we need $v_i + 2u_i   + w_i  \equiv 0 \pmod{q}$.
Furthermore, recall $v + 2u   + w  \equiv 0 \pmod{q}$. Since $u,v,w \le q-1$, we get that  $v + 2u   + w < 4q$.
If this sum equals $q$, it follows that $v + u n/q  + w (n-q)/q \le n$ and we are done, setting $v_2=u_2 = w_2 =0$.
So it remains to consider the cases that the sum is $2q$ or $3q$.

If the sum is $3q$, we note that $v+ 2u \ge 2q$. Let $u_1= u$ and $v_1 = 2q -2u$ and $w_1 = 0$.
Then $v_i + 2u_i   + w_i  \equiv 0 \pmod{q}$ and $v_i + u_i n/q  + w_i (n-q)/q \le n$.

If the sum is $2q$, then just let $v_1\le v$, $u_1\le u$, $w_1\le w$ such that $v_1 + 2u_1   + w_1  =q$; they clearly exist as $v,w$ are not (both) $0$. Then also $v_2 + 2u_2   + w_2  =q$. 
And since 
\[
v_i + u_i n/q  + w_i (n-q)/q \le n/q (v_i + 2u_i   + w_i),
\] 
the claim follows.
\end{proof}

We are now ready to prove Theorem \ref{main_thm_inverse}; some of the points are direct by the already obtained lemmas, some need additional work though.

\begin{proof}[Proof of Theorem \ref{main_thm_inverse}]
For each of the ten assertions, suppose that $n$ is such that the value for the minimal distance is integral -- otherwise the claim is vacuously true -- and suppose $G_0\subset G$ is a subset with the given minimal distance. We know by Theorem \ref{thm_directweak} that such a subset exists provided that the given value is integral.
We observe that by our assumption on $n$, we have that $n/5 \ge (2n^2)^{1/3}$.

We start with some general remarks.
By Theorem \ref{prop_2subset} there exists a subset
$G_2=\{f_1,f_2\} \subset G_0$ such that $\mD{G_2}>0$.
Suppose that $(\ord f_1,\ord f_2)$ is maximal
in the lexicographic order among all such subsets.
By Lemma \ref{aux_lem_higherorder} and \eqref{GeroHami},
we get that $\mD{G_2} \le (\ord f_2) -2$.

Further, by Lemma \ref{main_lem_ul} and Theorem \ref{2maximal} it follows that if $\mD{G_0}\notin \Hud{m}$ for some particular $m \mid n$, then every subset $\{f_1', f_2'\}$ of $G_0$ with $\gcd \{\ord f_1' , \ord f_2'\}= m$ is half-factorial.

Now, we give a definition useful for what follows. We say below that the minimal distance is contained \emph{only} in $\Hud{m_1}, \dots , \Hud{m_k}$, 
for some specified $m_i\mid n$, if it is not contained in $\Hud{m'}$ for each $m' \mid n$ not equal to one of the  $m_i$ while contained in $\Hud{m_i}$ for each $i$.

We start by observing that our lemmas allow to treat several cases directly. 
Specifically, Lemma \ref{lem_n-q-1q} with $q$ equal to $1$, $2$, $3$, $4$ yields (i), (iii), (v), (viii), respectively.
And, Lemma \ref{lem_n-2qq} with $q$ equal to  $2$ and $3$ yields (iv) and (vi), respectively. 

It thus remains to consider (ii), (vii), (ix), (x). 

\medskip

\noindent \emph{Proof of  (ii):} We have $\mD{G_0}=(n-2)/2$. This is only contained in $\Hud{n}$.
So, we have  $\ord f_1 = \ord f_2=n$. By Theorem \ref{2maximal} we get that $f_1 = -f_2$ and $G_0$ contains no other element of order $n$. The claim follows by Lemma \ref{smalldetailed}.

\medskip

\noindent \emph{Proof of  (vii):} We have $\mD{G_0}=(n-4)/4$. This is only contained in $\Hud{n}$ and $\Hud{n/2}$. Thus, we get that $(\ord f_1,\ord f_2) \in \{ (n,n),(n,n/2), (n/2,n/2) \}$.
We now distinguish these three cases.
\smallskip

a.) Suppose $(\ord f_1,\ord f_2)=(n,n)$.
By Theorem \ref{2maximal}, we get that $f_1 = ((n-2)/2)f_2$.
Moreover, by Lemma \ref{smalldetailed}, we get that
\[
G_0\subset \{f_1,f_2, 0, (n/2) f_1, 2f_1, 2 f_2\} = G_1.
\]
We assert that this set indeed has the right minimal distance.
Let $A$ be a minimal zero-sum sequence over $G_1$.
We need to show that $\sigma_{f_1}(A)/n - 1$ is a multiple of $(n-4)/4$.

We start by asserting that it suffices to consider $A$ containing
$f_1$ and $f_2$. Assume not, say it does not contain $f_2$. Since $(n-4)/4$ is integral we have that $(n/2)f_1 \in \langle 2f_1\rangle$. By Lemma \ref{aux_lem_higherorder},
the minimal distance of $\{f_1, 0, (n/2) f_1, 2f_1, 2 f_2\}$ is equal to  the minimal distance of $\{ 0, (n/2) f_1, 2f_1, 2 f_2\}$, which by Proposition \ref{prop_2maximal}, for $\langle 2 f_1 \rangle $ a group of order $n/2$, is $(n/2 -2)/2= (n-4)/4$, completing this argument.

We observe that if $A$ contains both $f_1$ and $f_2$ at least with multiplicity $2$, then $A= f_1^2f_2^2$ and  $\sigma_{f_1}(A)/n -1= 0$.
Without restriction we thus suppose that $A$ contains $f_2$ with multiplicity $1$. Consequently, it contains $f_1$ with an odd multiplicity, too (all other elements are contained in the subgroup $\langle 2 f_1 \rangle$).

Now, suppose $2f_i$ appears in $A$ for $i\in \{1,2\}$. Let $A'= f_i^2 (2f_i)^{-1}A$.
We claim that $A'$ is still a minimal zero-sum sequence.
This is the case as the multiplicity of $f_i$ in $A$ is odd, and thus each proper zero-sum subsequence of $A'$ would yield a proper zero-sum subsequence of $A$.

Observe that such a replacement does not affect $\sigma_{f_1}$, that is
$\sigma_{f_1}(A) = \sigma_{f_1}(A')$.
Now repeatedly applying this replacement, we obtain a minimal zero-sum sequence $A''$ with the same value under $\sigma_{f_1}$ that neither contains $2f_1$ nor $2f_2$ and is thus in fact a minimal zero-sum sequence over $\{f_1,f_2, 0, (n/2) f_1\}$. By Proposition \ref{prop_2maximal} we know that the minimal distance of this set $(n-4)/4$ and thus that $ \sigma_{f_1}(A'')/n -1$ is a multiple of $(n-4)/4$.

\smallskip

b.) Suppose $(\ord f_1,\ord f_2)=(n,n/2)$.
By Lemma \ref{aux_lem_higherorder} we get that the minimal distance of $\{f_1,f_2\}$ is equal to the minimal distance of $\{2f_1,f_2\}$. By Theorem \ref{2maximal} it follows that $2f_1 = - f_2$.
Moreover, it follows applying Lemma \ref{aux_lem_higherorder} and then Lemma \ref{smalldetailed} to $\{f_1\} \cup (G_0 \cap \langle 2f_1 \rangle)$, that
\[
G_0 \cap \langle 2f_1 \rangle \subset \{0,2f_1, -2f_1,  (n/2) f_1 \}.
\]
Let $h \in G_0 \setminus (\langle 2f_1 \rangle \cup \{f_1\})$.
We note that $\{f_1, h\}$ is half-factorial, i.e., $h= d' f_1$
for some $d' \mid n$; note that $\ord h \neq n$ and thus the minimal
distance cannot be a multiple of $(n-4)/4$ in any other way.

Moreover, it follows that
\[
2 d'f_1= 2h =  -f_2 = 2f_1
\]
or
\[
2d'f_1 =2h = d'' f_2= - 2d''f_1
\]
with some $d'' \mid n/2$. 

This means that $2d' -2$ or $2d'+2d''$ for some $d'' \mid n/2$ is a multiple of $n$. 

Since $2d' - 2$ is certainly not a multiple of $n$, we get that
$2(d' +d'')$ is a multiple of $n$. 
The only way in which this is possible is that that $d' = d'' = n/4$; to see this note that $d' \neq n/2$ as
$h \notin \langle 2f_1\rangle$ and $4\mid n$.
Moreover, we observe that this can only happen if $8 \nmid n$, since otherwise $h \in \langle 2f_1  \rangle $.

Yet, considering the minimal zero-sum sequence
\[
f_1\ ((n/4) f_1)\ ((n-2)f_1)^{(n+4)/8},
\]
and $\sigma_{f_1}$ of it, it follows that $(n/4) f_1 \notin G_0$.
Thus
\[
G_0 \subset \{f_1,2f_1, -2f_1, (n/2)f_1, 0 \}\quad  \text{ and }\quad
\{f_1, -2f_1, (n/2)f_1\} \subset G_0;
\]
note that we must have $(n/2)f_1$ in $G_0$ since otherwise the minimal distance is $n/2 - 2$.
\smallskip

c.) Suppose $(\ord f_1,\ord f_2)=(n/2,n/2)$. We get by Theorem \ref{2maximal} applied to a group of order $n/2$ that $f_1 = - f_2$.
Moreover, it follows that
\[
G_0 \cap \langle f_1 \rangle \subset \{0, (n/4) f_1, f_1, -f_1\}.
\]
Let $h \in G_0 \setminus \langle f_1 \rangle$.
It follows that $2h\in \{0, (n/4) f_1, f_1, -f_1\}$.
Since $h$ cannot have order $n$, since this would contradict our assumption on the maximality of $(\ord f_1,\ord f_2)$,  or $2$, since the element of order $2$ is contained in the subgroup generated by $\langle f_1 \rangle$, it follows that $2h = (n/4)f_1$.
That is $h$ is an element of order $4$; note that this is only possible
if $8\nmid n$, since otherwise all element of order $4$ would be contained in $\langle f_1 \rangle$. Of course, $G_0$ contains at most one element of order $4$, as the minimal distance of a set of two elements of order $4$ is $2$.
Thus, it follows that
\[
G_0 \subset \{2f, -2f, (n/4) f,(n/2) f,0 \}
\]
for some generating element $f$.
\medskip

\noindent \emph{Proof of  (ix):} We have $\mD{G_0}=(n-6)/4$. This is only contained in $\Hud{n/2}$.
So, we get that $(\ord f_1,\ord f_2) \in \{ (n,n/2), (n/2,n/2) \}$.
\smallskip

a.)
First, suppose $(\ord f_1,\ord f_2)=(n/2,n/2)$. It follows by Theorem \ref{2maximal} and symmetry that we may assume
$f_2 = -2f_1$.
And, moreover, by Lemma \ref{smalldetailed},  $G_0 \cap \langle f_1 \rangle \subset \{f_1,f_2, 0\}$.
Let $h\in G_0 \setminus \langle f_1 \rangle$.
It follows by Lemma \ref{aux_lem_higherorder} that $2h \in \{f_1,f_2, 0\}$. Yet, $\ord h \neq n$, since this would contradict the maximality of $(\ord f_1,\ord f_2)$  and thus
$2h = 0$. Thus,
\[
G_0 \subset \{2f, -4f, (n/2) f, 0\}
\]
for some generating element $f$. By Lemma \ref{aux_lem_higherorder} the minimal distance of this set is equal to the one of $\{2f, -4f, 0\}$ and that set actually has the
required minimal distance by Proposition \ref{prop_2maximal}.
\smallskip

b.) Now, suppose  $(\ord f_1,\ord f_2)=(n,n/2)$. It follows by Lemma \ref{aux_lem_higherorder} and Theorem \ref{2maximal} for a group of order $n/2$ that
$f_2 = -4 f_1$ or $f_2 =  ((n-2)/2)  f_1$.
Moreover, by Lemma \ref{smalldetailed}, $G_0 \cap \langle f_2 \rangle \subset \{2f_1, f_2 , 0\}$.
Let $h\in G_0 \setminus (\langle f_2 \rangle \cup \{f_1\})$.
It follows that $2h \in  \{2f_1, f_2 , 0\} $. Since $\ord h \neq n$,
it follows that $h = (n/2) f_1$. Thus
\[
G_0 \subset \{f_1 , 2f_1, -4f_1 , (n/2) f_1, 0 \} \text{ or }
G_0 \subset \{f_1 , 2f_1, ((n-2)/2) f_1 , (n/2) f_1, 0 \}.
\]

It remains to check whether or not the sets have the required
minimal distance. We first consider
$\{f_1 , 2f_1, -4f_1 , (n/2) f_1, 0 \}$.
We consider the minimal zero-sum sequence
\[
f_1^x\ ((n/2)f_1)\ (-4f_1)^{(n+2x)/8}
\]
where $x\in \{1,3\}$ depending on
the congruence class of $n$ modulo $8$.
In any case, its value under  $\sigma_{f_1}(\cdot)/n -1$ is not
a multiple of $(n-6)/4$, showing that in this case we cannot have
$(n/2) f\in G_0$.
\smallskip

Now, we consider $\{f_1 , 2f_1, ((n-2)/2) f_1 , (n/2) f_1, 0 \}$.
Since $\{f_1 , 2f_1, (n/2) f_1, 0 \}$ is half-factorial, and both  $\{ 2f_1, ((n-2)/2) f_1 , (n/2) f_1, 0 \}$ and $\{f_1 , 2f_1, ((n-2)/2) f_1 , 0 \}$ have the required minimal distance, it suffices to consider minimal zero-sum sequences containing $f_1$, $((n-2)/2) f_1$,
and $(n/2) f_1$. The only one  is however
\[
f_1\ (((n-2)/2) f_1)\ ((n/2) f_1),
\]
and its value under $\sigma_{f_1}(\cdot)/n -1$ is $0$.
It follows that the set has the required minimal distance.
\medskip

\noindent \emph{Proof of  (x):} We have $\mD{G_0}=(n-8)/4$. This is only contained in $\Hud{n/4}$.
So, we get that $(\ord f_1,\ord f_2) \in \{ (n,n/4), (n/2,n/4), (n/4,n/4) \}$.
\smallskip

a.) Suppose that $(\ord f_1,\ord f_2)=(n/4,n/4)$.
It follows by Theorem \ref{2maximal} that $f_1 = -f_2$. We note that $G_0 $ cannot contain
an element of order $n$ or $n/2$; if it would contain such an element,
together with $f_1$ or $f_2$, we get a set violating our assumption.
Moreover, by Lemma \ref{smalldetailed},  $G_0 \cap \langle f_1 \rangle \subset \{f_1, -f_1, 0\}$.
Let $h \in G_0 \setminus \langle f_1 \rangle$.

Let $a$ be minimal such that $ah \in \langle f_1 \rangle$;
note that $a\in \{2,4\}$. It follows that $ah= 0$. Clearly $G_0$
cannot contain two elements of order $4$; compare the argument in (vii).
It follows that
\[
G_0 \subset \{4f, -4f, 0,(n/2) f, (n/4)f\}.
\]
And, note that we can only have $(n/2)f$ or $(n/4)f$ in
$G_0$ if $8 \nmid n$.

By Lemma \ref{aux_lem_higherorder}, applied twice, we get that the minimal distance of the set $\{4f, -4f, 0,(n/2) f, (n/4)f\}$ is $(n-8)/4$, and the claim follows.

\smallskip

b.) Suppose that $(\ord f_1,\ord f_2)=(n/2,n/4)$.
It follows by Lemma \ref{aux_lem_higherorder} and Theorem \ref{2maximal} that $f_2 = -2f_1$.
Moreover, by Lemma \ref{smalldetailed}, $G_0 \cap \langle f_2 \rangle \subset \{f_2, -f_2, 0\}$.
Let $h \in G_0 \setminus  (\langle f_2 \rangle \cup \{f_1\})$.
First, suppose $h \in \langle f_1 \rangle$, that is
$2h \in \langle f_2 \rangle$. It follows that $2h = f_2$ or
$2h = -d'f_2 = - 2d'f_1$ for some $d' \mid n/4$.
Yet, considering $\{f_1, h\}$ it follows that $h= d''f_1$
for some $d'' \mid n/2$ so that the set is half-factorial; 
note that  $2h = -2f_1$ is impossible.
Thus, $2(d' + d'')$ is a multiple of $n/2$. This implies that $h= (n/4) f_1$
or $h= (n/8)f_1$. Let $d$ equal $n/8$ or $n/4$, respectively.
We consider the minimal zero-sum sequence
\[
A= f_1\ (df_1)\ (-2f_1)^{(d+1)/2}
\] 
(note that $d$ is odd as
$h \notin \langle f_2 \rangle$).
We compute $\sigma_{f_1}(A)/(n/2) -1 = (d-1)/2$ (note the $n/2$),
and see that this is not a multiple of $(n-8)/4$.

Second, suppose $h \notin \langle f_1 \rangle$.
It follows that $2h = d'' f_1$ for some $d''\mid n/2$; note that
$2h \neq - 2f_1$. Moreover, $4h = d'f_2 = -2d' f_1 $ for some $d' \mid n/4$.
This implies that $2h = (n/4) f_1$ or $2h = (n/8) f_1$,
yielding a contradiction as above.

Thus, $G_0 \subset \{f_1, 2f_1, -2f_1, 0\}$; by Lemma \ref{aux_lem_higherorder} and Proposition \ref{prop_2maximal} the latter set has the required minimal distance. Note that $G_0$ is of the claimed form with $j=2$.

\smallskip

c.) Suppose that $(\ord f_1,\ord f_2)=(n,n/4)$.
Again, it follows that $f_2 = -4f_1$ and $G_0 \cap \langle f_2 \rangle \subset \{f_2, -f_2, 0\}$.
Let $h \in G_0 \setminus  (\langle f_2 \rangle \cup \{f_1\})$.
It follows that $h= d'f_1$ with $d'\mid n$.

Let $b\in \{2,4\}$ minimal such that $bh \in \langle f_2 \rangle$.
It follows that $bh = -f_2$ or $bh = d''f_2$ for some $d'' \mid n/4$.
In the former case it follows that $h= 2f_1$.

In the latter case it follows that $b d' + 4 d''$ is a multiple of $n$.
For $b=4$, this implies that $d' \in \{n/4,n/8\}$, and  for
$b=2 $, this implies $d' \in \{n/2, n/4\}$.

Let $x\in \N$ be minimal such that $4 \mid x +d'$.
We consider the minimal zero-sum sequence
\[
A= f_1^x\ (d'f_1)\ (-4f_1)^{(d+x)/4}.
\]
Since $\sigma_{f_1}(A)/n -1 = (d' +x-1)/4$, is not a multiple of $(n-8)/4$,
such an $h$ cannot exist.

Thus, $G_0 \subset \{f_1, 2f_1, 4f_1, -4f_1, 0\}$ and the latter set has the required minimal distance by Lemma \ref{aux_lem_higherorder}, applied twice, and Theorem \ref{2maximal}.
\end{proof}

The following result is mainly intended to further investigate the possibility of extending Theorem \ref{main_thm_direct}. 

\begin{lemma}
\label{lem_n-410}
Let $G$ be a finite cyclic group and $G_0 \subset G$.
For $|G| \ge  2000$ we have $\mD{G_0} \neq (|G|-4)/10$.
\end{lemma}
\begin{proof}
Evidently we only need to consider the case that $(|G|-4)/10$
is an integer. Note that $(|G|-4)/10$ is contained in $\Hud{|G|/2}$ but no other set $\Hud{m}$ for $m \mid |G|$. By Theorem \ref{prop_2subset}, we get that there exists a non-half-factorial subset $G_2$ of cardinality two, whose minimal distance is  $(|G|-4)/2$. By Lemma \ref{main_lem_ul} and Theorem \ref{2maximal} we know that $G_2 = \{f, -2f\}$ or
$G_2 = \{2f , -2f\}$ for some generating element $f$ of $G$.

First, suppose the latter is the case and $G_0$ does not contain a subset of the former form. By Lemma \ref{smalldetailed} applied to the group $\langle 2f \rangle $ it follows that $G_0 \cap \langle 2f \rangle \subset \{ 2f, -2f, 0 \}$. Moreover, it follows that for $h \in G_0 \setminus \langle 2f \rangle$
 we have $2h \in \{ 2f, -2f, 0 \}$. By our assumption that $G_0$ does not contain a subset of the form $\{f, -2f\}$, we get that $2h=0$. Thus $G_0 \setminus \langle 2f \rangle$ contains at most one element, and the minimal distance of $G_0$ is $(|G|-4)/2$, a contradiction.

Second, suppose $G_0$ contains $ \{f, -2f\}$. We note that by Lemma \ref{aux_lem_higherorder} and Lemma \ref{smalldetailed} $\langle 2f\rangle \cap G_0 \subset \{0, 2f , -2f\}$.

Suppose there exists an element $h \in G_0 \setminus (\{f\} \cup \langle 2f \rangle )$.
It follows that $\{f, h\}$ is half-factorial and thus $h= d f$  for some odd $d \mid n $. Moreover, it follows that $\{2df, -2f\}$ is half-factorial, too. This implies that $2df=-2df$, that is $n \mid 4d$. So, $d \in \{n/4,n/2\}$.
We consider $A= f (df) (-2f)^{(d+1)/2}$, which is a minimal zero-sum sequence. Since $\s_{f_1}(A) = (d+1)/2  n $, it follows by Lemma \ref{aux_lem_geroldinger} that $\mD{G_0}\mid (d-1)/2$, a contradiction to $\mD{G_0} = (n-4)/10$.

Consequently, we get that $ G_0\subset \{f, 2f, -2f, 0\}$, and so its minimal distance is $(n-4)/2$ by Proposition \ref{prop_2maximal}, again a contradiction.
\end{proof}

\section{Applications to congruence half-factorial structures}
\label{sec_appCHF}

In this section we detail how our results can be applied to questions related to congruence half-factorial structures. To some extent this was already discussed at the beginning of the paper. In view of known results, the results below are fairly direct consequences of the results in the preceding sections; however, there are some subtleties that we believe are worthwhile to be stressed.

As mentioned in Section \ref{sec_prel} for a Krull monoid $H$ it is well known that
\[\lc(H) = \lc(\bc(G_0))\]
where $G_0$ is the subset of classes containing prime divisors.
In view of the fact recalled in Section \ref{ssec_kn} that $d$-congruence half-factoriality is equivalent to $d$ dividing $\min \Delta(H)$ and so $d$ dividing $\min\Delta(G_0)$ a close link between the present and preceding problems is evident.
There is however one additional point to consider. Namely, for a subset $G_0 $ of $G$ we need to know that there actually exists a Krull monoid having that class group and that subset of classes containing prime divisors.

This has one, but only one, implication for the set $G_0$.
We recall the relevant result (see \cite[Theorem 2.5.4]{geroldingerhalterkochBOOK}).

\begin{theorem}
\label{thm_KM_ex}
Let $G$ be an abelian group and $G_0 \subset G$.
There exists a Krull monoid with class group isomorphic to $G$ such that the set of classes containing prime ideals corresponds to $G_0$ if and only if $G_0$  generates $G$ as a semi-group.
\end{theorem}

In case $G$ is a torsion group (with at least two elements), the condition that $G_0$  generates $G$ as a semi-group is of course  equivalent to the condition that $G_0$ generates $G$ as a group.

Such a result not only holds for Krull monoids but even for Dedekind domains;  
we refer to \cite{gilmeretal96} for a refined version mainly concerned with the number of prime divisors in the classes in the domain case. By contrast for Krull monoids there is less restriction on the number of prime divisors in each class.  That there is a difference between the domain and monoid case is actually a somewhat rare phenomenon in this context. In possibly more classical terms, one can understand this difference by recalling that Dedekind domains have the approximation property while a Krull monoid might not have it.
An arithmetic result where this difference is visible and crucial is the work of Coykendall and Smith \cite{coykendallsmith11} on `other half-factorial structures' (a notion we do not recall here).

A half-factorial structure is $d$-congruence half-factorial for each $d$. Since it is well-known that half-factorial Krull monoids with finite cyclic class group of any order exist (see for example Section \ref{non-LCN}), the mere question for which $d$ it is true that $d$-congruence half-factorial Krull monoids exist, does not make much sense.
Yet, restricting to non-half-factorial structures it becomes interesting. An additional natural question is to ask for the $d$ such that there exists a $d$-congruence half-factorial Krull monoid which is not $d'$-congruence half-factorial for any multiple of $d'$; in this case, we say that the monoid is \emph{truly $d$-congruence half-factorial}.

The following result is not surprising and it might even seem obvious, yet there is one subtle point making it in the end not as obvious as it might seem. Indeed, we do not know how to prove the result for general finite groups and it is not clear whether it holds; note that we use Lemma \ref{aux_generating}.

\begin{theorem}
Let $d$ and $n$ be positive integers.
\begin{enumerate}
\item There exists a non-half-factorial truly $d$-congruence half-factorial Krull mo\-noid with finite cyclic class group of order $n$ if and only if $d\in \Dast{C_n}$,
\item There exists a non-half-factorial $d$-congruence half-factorial Krull monoid with finite cyclic class group of order $n$ if and only if $d$ divides an element of $\Dast{C_n}$.
\end{enumerate}
\end{theorem}
These results as well as the subsequent ones are also true for Dedekind domains instead of Krull monoids, as can be seen from the discussion above.

\begin{proof}
We start by proving the first part. Suppose $H$ is a truly $d$-congruence half-factorial structure with the required properties. And, let $G_0$ be the subset of ideal classes containing prime divisors. It is known (cf.~Section \ref{ssec_kn}) that $d\mid  \min \Delta(G_0)$ and that the structure is $\min \Delta(G_0)$-congruence half-factorial. This establishes one part of the proof.

Now, suppose that $d\in \Dast{C_n}$. Merely from the definition, it follows that there exist a subset $G_0\subset C_n$ such that $\min \Delta(G_0)=d$ and thus $\bc(G_0)$ is truly $d$-congruence half-factorial.
However, this is not sufficient to obtain our claim by Theorem \ref{thm_KM_ex}, since we have no guarantee that $G_0$ is a generating set. Yet, by Lemma \ref{aux_generating} there exists a generating set $G_0'\subset C_n$ such that $\mD{G_0'}=\mD{G_0}$ and the claim follows.

We now turn to the second part. Suppose $H$ is $d$-congruence half-factorial, it follows that there exists some $d'$ such that $H$ is truly $d'$-congruence-half-factorial. By the first part, we know that $d'\in \Dast{C_n}$ and the claim follows. Conversely, if $d$ divides an element of $\Dast{C_n}$, denote this element by $d'$, we get by the first part that there exist a truly $d'$-congruence-half-factorial Krull monoid. Now, this monoid is $d$-congruence half-factorial, establishing the claim.
\end{proof}

Evidently, this result can now be combined with the results on $\Dast{G}$ obtained in Section \ref{sec_proofs}, or any other result for $\Dast{G}$ for finite cyclic groups, to obtain more `explicit' versions.
We only phrase one such result, since due to a fortunate coincidence it is much stronger than what one might expect in view of the results of Section \ref{sec_proofs}.

\begin{theorem}
Let $d$ and $n$ be integers with $d > n^{1/2}$.
There exists a non-half-factorial $d$-congruence half-factorial Krull monoid with finite cyclic class group of order $n$ if and only if
\[
d \in \bigcup_{m \mid n} \Hud{m}.
\]
\end{theorem}
\begin{proof}
Let $d$ be of the form given in the result. This means that $m\mid n$ and that $c_1,c_2$ are integers such that
\[
\frac{m - c_1 - c_2}{c_1 c_2}
\]
is integral and a multiple of $d$. Let $G$ be a finite cyclic group of order $n$, and $e$ a generating element of $G$.
The set $G_0=\{e, (n/m)e , \frac{m - c_1}{c_2} ((n/m)e) \}$ has, by Lemma \ref{aux_lem_higherorder}, the same minimal distance as 
\[
\{(n/m)e , \frac{m - c_1}{c_2} ((n/m)e) \}.
\]
By Theorem \ref{aux_thm_2elements}, applied to $\langle (n/m)e \rangle$, a cyclic group of order $m$, it is $(m-c_1 -c_2)/(c_1c_2)$.
By Theorem \ref{thm_KM_ex} there exists a Krull monoid with class group isomorphic to $G$ such that the subset of classes containing prime divisors is $G_0$. By standard results recalled in Section \ref{ssec_kn} the minimal distance of this Krull monoid is the minimal distance of $G_0$. And, thus it is $d$-congruence half-factorial.

To see the converse claim suppose that a certain Krull monoid with class group $G$, a finite cyclic group of order $n$, is $d$-congruence half-factorial. We know that its minimal distance $d'$ is a multiple of $d$.
Since $d'$ is non-zero we also have that $d' \ge \log n$.
Again by Section \ref{ssec_kn} we know that $d'$ is equal to the minimal distance of $G_0\subset G$ the subset of classes containing prime divisors. Now, by Theorem \ref{prop_2subset} we know that there exists a non-half-factorial subset $G_2 \subset G_0$ of cardinality two.
By Lemma \ref{aux_lem_higherorder} we may assume without restriction that the two elements in $G_2$ have the same order $m \mid n$. Now, it follows by Theorem \ref{aux_thm_2elements} that $d'$ is of the form $(m-c_1 -c_2)/(c_1c_2)$ establishing the claim.
\end{proof}

We now address the other question, that is we give a precise characterization of Krull monoids with finite cyclic class group of order $n$ that are  $d$-congruence half-factorial or truly $d$-congruence half-factorial Krull monoids, for $d\ge n/5$ in the general case and for $d\ge (2n^2)^{1/3}$ for $n$ prime. The result follows quite directly from Theorems \ref{main_thm_prime} and \ref{main_thm_inverse}.

\begin{theorem}
\label{char_thm}
Let $H$ be a non-half-factorial Krull monoid with finite cyclic class group of order $n$.
Suppose $d$ is one of the elements appearing in Theorem \ref{main_thm_inverse} (and $n$ is such that $d$ is integral). Then,
\begin{enumerate}
\item $H$ is truly $d$-congruence half-factorial if and only if the subset of classes containing prime divisors is equal to a generating set fulfilling the condition for this $d$ in Theorem \ref{main_thm_inverse},
\item $H$ is $d$-congruence half-factorial if and only if the subset of classes containing prime divisors is equal to a generating set fulfilling the condition for a multiple of this $d$ in Theorem \ref{main_thm_inverse}.
\end{enumerate}
\end{theorem}
\begin{proof}
Let $G$ denote the class group and $G_0$ the subset of classes containing prime divisors.
By Section \ref{ssec_kn} we know that $H$ is  $d$-congruence half-factorial if and only if  $d$ is a divisor of $\mD{G_0}$. 
Likewise, $H$ is truly $d$-congruence half-factorial if and only if $d$ equals $\mD{G_0}$. 
By Theorem \ref{thm_KM_ex} we know that $G_0$ is a generating set of $G$.
Now, the claim follows by Theorem \ref{main_thm_inverse}.
\end{proof}

One could extend the second part of this result  to the case of those $d$ for which one knows by results on $\Dast{G}$, such as Theorem \ref{main_thm_direct}, that all multiples of $d$ that are in $\Dast{G}$ are at least of size $n/5$, and thus covered by Theorem \ref{main_thm_inverse}.

For completeness, we end this section by giving the proofs of the two results mentioned at the beginning of the paper, which are now extremly short.

\begin{proof}[Proof of Theorem \ref{thm_caseprime}]
The proof is analogous to the one of Theorem \ref{char_thm} except for applying Theorem \ref{main_thm_prime} instead of Theorem \ref{main_thm_inverse}.
\end{proof}

\begin{proof}[Proof of Theorem \ref{thm_genprel}]
This is just a special case of Theorem \ref{char_thm}.
\end{proof}

\section{Number theoretic applications}
\label{sec_nt}

In this final section, we discuss some applications to questions of quantitative problems on  factorizations, specifically in rings of algebraic integers, yet they apply verbatim in other or more general context as well (holomorphy rings of function fields and formations, resp., see the introduction to Chapter 9 of \cite{geroldingerhalterkochBOOK}).

First we recall the abstract description for $\ao$ and $\bo$ mentioned in Section \ref{ssec_kn}; we refer to \cite[Theorem 9.4.10]{geroldingerhalterkochBOOK} and surrounding results for details.
It is known that $\ao$ is equal to the maximal cardinality
of certain subsets of $G$.
We recall a notation: for a set $G_0 \subset G$, a (possibly empty) sequence  $S\in \fc(G \setminus G_0)$, and $\ell \in \N_0$, let
\[
\Omega(G_0,S,\ell)
\]
denote the set of all zero-sum sequences $SF$ such that $F \in \fc(G_0)$ with $\vo_g(F)\ge \ell $ for each $g \in G_0$.

Then  $\ao$ is the maximum over all $G_0 \subset G$ such that for some $S\in \fc(G \setminus G_0)$ and some $\ell \in \N_0$ one has
 \begin{equation}
\label{eq_PHDM}
\emptyset \neq \beta^{-1}(\Omega(G_0,S, \ell)) \subset \pc(H, \dc, M).
\end{equation}
For clarity, note that the condition that the set is nonempty merely means that $\s(S) \in \langle G_0 \rangle$.
We denote $\ao$ by $\ao_{\dc}(G)$.
For completeness, we recall that $\bo$ is the maximal length of a sequence $S$ for which \eqref{eq_PHDM} holds for some $G_0$ with maximal cardinality, i.e., cardinality $\ao_{\dc}(G)$.

While this cannot be the  place to recall how these results are obtained in any detail, we still give some very rough indications, in the hope that they clarify a bit why the constants are of this form.

One can show $\pc(H, \dc, M)$ is a \emph{finite} union of sets  of the form $\beta^{-1}(\Omega(G_0,S, \ell))$.
The order of the counting function of these special sets can be determined via expressing the associated Dirichlet series as a suitable combination of $L$-series, and then applying an appropriate Tauberian theorem.
The order of the counting function of a set $\beta^{-1}(\Omega(G_0,S, \ell))$ is 
\[
\frac{x}{(\log x)^{1- |G_0|/|G|}} (\log \log x)^{|S|-\varepsilon}
\] 
with $\varepsilon$ equal to $0$ or $1$ depending on whether $G_0$ is non-empty or empty. 
Finally, note that for the dominant terms the sets $G_0$ are non-empty.

In the following result we determine $\ao_{\{0,d\}}(G)$ for finite cyclic $G$ and large $d$; the case $d= n-2$ appeared already in \cite{WAS16}.

\begin{theorem}
Let $G$ be a finite cyclic group of order $n\ge 250$.
\begin{enumerate}
\item If $d \in \N \cap \{n-2, (n-3)/2, (n-4)/3, (n-5)/4\}$,
then $\ao_{\{0,d\}}(G)=3$,
\item If $d \in \N \cap \{(n-2)/2, (n-4)/2, (n-6)/3\}$,
then $\ao_{\{0,d\}}(G)=4$,
\item If $d \in \N\cap \{(n-6)/4,(n-8)/4\}$, then $\ao_{\{0,d\}}(G)=5$,
\item If $d=(n-4)/4$ is integral, then $\ao_{\{0,d\}}(G)=6$.
\end{enumerate}
\end{theorem}

\begin{proof}
From the description of $\ao_{\dc}(G)$ recalled above 
one can derive, using well-known arguments, that for $\dc = \{0,d\}$ with $d \in \Dast{G}$ one has 
\[
\begin{split}
\max \{|G_0|  \colon \min \Delta(G_0) = d , \, G_0 \subset & G\}  \leq \\
 & \ao_{\{0,d\}}(G) \leq \max \{|G_0| \colon d \mid \min \Delta(G_0) , \, G_0 \subset G \}.
\end{split}
\]

Now, from Theorem \ref{main_thm_inverse} we see that in all cases
the maximum cardinality on the right hand side is actually attained
for $d$ (as opposed to a proper multiple) and we thus get an equality,
and the exact value.
\end{proof}

\begin{theorem}
Let $G$ be a finite cyclic group of order $n$ and let $q\mid n$ be an odd prime. 
\begin{enumerate}
\item If $d = (n-q-1)/q$ is integral and  $q \le n^{1/3}/2$, then $\ao_{\{0,d\}}(G)=3$,
\item If $d = (n-2q)/q$ is integral and  $q \le n^{1/3}/2$, then $\ao_{\{0,d\}}(G)=4$,
\item If $d = (n-2q)/q^2$ is integral and  $q \le n^{1/6}/2$ , then $\ao_{\{0,d\}}(G)=6$. 
\end{enumerate}
\end{theorem}
\begin{proof}
The argument is identical to the one of the preceding result, except that we use Lemmas \ref{lem_n-q-1q}, \ref{lem_n-2qq}, and \ref{lem_n-2qq2} instead of Theorem \ref{main_thm_inverse}.
\end{proof}

Moreover, it is not hard to see that the values of $\ao_{\{0,d\}}(G)$, for the $d$ we considered, are upper bounds for $\ao_{ \dc }(G)$ for every $\{ 0, d \} \subset \dc \subset [0,d]$ such that the period $ \dc $ is aperiodic, i.e., there exists no $x \in \Z \setminus d\Z$ such that the image of $\dc$ and $x + \dc$ in $ \Z / d \Z$ are equal.

It could be interesting to pursue these ideas further, including an analysis of $\bo_{\dc}$, yet we do not do so here.

\bibliographystyle{abbrv}


\end{document}